\documentclass[]{article}

\usepackage{graphicx}
\usepackage{amsmath}
\usepackage{amssymb}
\usepackage{amsthm} 
\usepackage{bm}
\usepackage{pxfonts}
\usepackage{enumerate}
\usepackage{color}
\usepackage{mathdots}
\usepackage{sectsty}
\usepackage[hidelinks]{hyperref}
\usepackage{tikz}
\usepackage{caption}
\usepackage{adjustbox}
\usepackage{bbold}

\sectionfont{\scshape\centering\fontsize{11}{14}\selectfont}
\subsectionfont{\scshape\fontsize{11}{14}\selectfont}
\usepackage{fancyhdr}

\newcommand\shorttitle{On a distinguished family of random variables and Painlev\'e equations}
\newcommand\authors{T. Assiotis, B. Bedert, M. A. Gunes and A. Soor}

\newtheorem{thm}{Theorem}[section]
\newtheorem{cor}[thm]{Corollary}

\newtheorem{defn}[thm]{Definition}
\newtheorem{rmk}[thm]{Remark}
\newtheorem{prop}[thm]{Proposition}
\newtheorem{conjecture}[thm]{Conjecture}

\newtheorem*{claim*}{Claim}

\DeclareMathOperator{\sgn}{sgn}

\title{\large \bf On a Distinguished Family of Random Variables and Painlev\'e Equations}
\author{\small THEODOROS ASSIOTIS, BENJAMIN BEDERT, MUSTAFA ALPER GUNES AND ARUN SOOR}
\date{}

\begin{document}

\maketitle
\begin{abstract}
A family of random variables $\mathbf{X}(s)$, depending on a real parameter $s>-\frac{1}{2}$, appears in the asymptotics of the joint moments of characteristic polynomials of random unitary matrices and their derivatives \cite{assiotispainleve}, in the ergodic decomposition of the Hua-Pickrell measures \cite{BorodinOlshanskiErgodic}, \cite{Qiu} and conjecturally in the asymptotics of the joint moments of Hardy's function and its derivative \cite{Zeta4}, \cite{assiotispainleve}. Our first main result establishes a connection between the characteristic function of $\mathbf{X}(s)$ and the $\sigma$-Painlev\'e III' equation in the full range of parameter values $s>-\frac{1}{2}$. Our second main result gives the first explicit expression for the density and all the complex moments of the absolute value of $\mathbf{X}(s)$ for integer values of $s$. Finally, we establish an analogous connection to another special case of the $\sigma$-Painlev\'e III' equation for the Laplace transform of the sum of the inverse points of the Bessel point process.
\end{abstract}

\tableofcontents

\section{Introduction}
\subsection{Motivation}
We first give the precise definition of the random variables $\mathbf{X}(s)$, with $s>-\frac{1}{2}$, as principal value sums of certain determinantal point processes. We then elaborate on three distinct reasons why someone would be interested in them.
\begin{defn}
Let $s\in \mathbb{R}$ and $s>-\frac{1}{2}$. Let $\mathbf{C}^{(s)}$ be the determinantal point process\footnote{By a determinantal point process $\mathfrak{X}$ on $E\subseteq \mathbb{R}$ with correlation kernel $\mathfrak{L}:E\times E \to \mathbb{C}$ we mean a random point process $\mathfrak{X}$ on $E$, or equivalently a probability measure $\mathbb{P}^{\mathfrak{X}}$ on the space of locally finite point configurations $\mathsf{Conf}(E)$ on $E$, satisfying for all $n\in \mathbb{N}$ and any compactly supported bounded Borel function $\mathsf{F}$ on $E^n$:
\begin{equation*}
\int_{\mathsf{Conf}(E)}\sum_{x_{i_1},\dots,x_{i_n}\in \mathcal{X}}\mathsf{F}(x_{i_1},\dots,x_{i_n})\mathbb{P}^{\mathfrak{X}}(d\mathcal{X})=\int_{E^n}\mathsf{F}(y_1,\dots,y_n)\det\left[\mathfrak{L}(y_i,y_j)\right]_{i,j=1}^{n} dy_1 \cdots dy_n,
\end{equation*}where the sum is over all $n$-tuples of pairwise distinct points of the point configuration $\mathcal{X}\in \mathsf{Conf}(E)$, see \cite{Borodin}.} on $\mathbb{R}^*=(-\infty,0) \cup (0,\infty)$ with correlation kernel:
\begin{align*}
K^{(s)}(x,y) = \frac{1}{2\pi}\frac{\left(\Gamma(s+1)\right)^2}{\Gamma(2s+1)\Gamma(2s+2)}\frac{T^{(s)}(x)R^{(s)}(y)-T^{(s)}(y)R^{(s)}(x)}{x-y},
\end{align*}
where $T^{(s)}(x), R^{(s)}(x)$ are given by the formulas 
\begin{equation*}
T^{(s)}(x) = 2^{2s-\frac{1}{2}}\Gamma\left(s+\frac{1}{2}\right) \cdot \frac{1}{|x|^{\frac{1}{2}}}J_{s-1/2}\left(\frac{1}{|x|}\right),
\end{equation*}
\begin{equation*}
R^{(s)}(x) = 2^{2s+\frac{1}{2}}\Gamma\left(s+\frac{3}{2}\right) \cdot \frac{1}{|x|^{\frac{1}{2}}} J_{s+1/2}\left(\frac{1}{|x|}\right),
\end{equation*}
where $J_{\nu}$ denotes the Bessel function with parameter $\nu$. 
Then, $\mathbf{X}(s)$ is defined to be the following principal value sum of the points of $\mathbf{C}^{(s)}$, shown to be well-defined by the results of Qiu \cite{Qiu}:
\begin{equation}
    \mathbf{X}(s)=\lim\limits_{m\to\infty}\left[\sum_{x\in \mathbf{C}^{(s)}}x\mathbb{1}\left(|x|>\frac{1}{m^2}\right)\right].
    \label{eq:Xsdefinition}
\end{equation}

\end{defn}

\paragraph{Joint moments of characteristic polynomials of random unitary matrices.} Let $\mathbf{U}\in \mathbb{U}(N)$ be a Haar-distributed random matrix, where $\mathbb{U}(N)$ is the group of $N\times N$ unitary matrices, and let  $e^{i\theta_1}, \ldots,e^{i\theta_N}$ denote its eigenvalues.

Define the characteristic polynomial of $\mathbf{U}$:
\begin{equation}
    \mathbf{S}_{\mathbf{U}}(\theta)=\det\left(I-e^{-i\theta}\mathbf{U}\right)
\end{equation}
and consider:
\begin{equation}
\mathbf{G}_{\mathbf{U}}(\theta)=e^{\frac{iN}{2}(\theta+\pi)-i\sum_{k=1}^N\frac{\theta_k}{2}}\mathbf{S}_{\mathbf{U}}(\theta),
\end{equation}
so that $|\mathbf{S}_{\mathbf{U}}(\theta)|=|\mathbf{G}_{\mathbf{U}}(\theta)|$ and $\mathbf{G}_{\mathbf{U}}(\theta)$ is real-valued for $\theta \in [0,2\pi)$.

Let $s\in \mathbb{R}$ and $s>-\frac{1}{2}$. Then we define, for $-\frac{1}{2}<h<s+\frac{1}{2}$, the following quantities, that we call the joint moments:
\begin{equation}
    \mathcal{R}_N(s,h)=\int_{\mathbb{U}(N)} \left|\mathbf{G}_{\mathbf{U}}(0)\right|^{2s-2h} \left|\frac{d\mathbf{G}_{\mathbf{U}}}{d\theta}\Bigr\rvert_{\theta = 0}\right|^{2h} d\mu_N(\mathbf{U}),
    \label{eq:jointmoments}
\end{equation}
where $d\mu_N(\mathbf{U})$ is the Haar probability measure on the group of unitary matrices $\mathbb{U}(N)$. 

Hughes, in his thesis \cite{Zeta4} from 2001, partly motivated by connections with number theory that we will say more about below, made the following conjecture about the asymptotics of the joint moments:
\begin{equation}
      \frac{\mathcal{R}_N(s,h)}{N^{s^2+2h}} \overset{?}{\underset{N\to \infty}{\longrightarrow}} \mathcal{R}(s,h), \label{eq:Rfunction}
\end{equation}
for some unidentified (for generic real values of the exponents) quantity $\mathcal{R}(s,h)$. The conjecture was proven for $s\in \mathbb{N}$ \footnote{In this paper we use the convention $0 \notin \mathbb{N}$.} and $h \in \mathbb{N}$ or $h\in \mathbb{N}-\frac{1}{2}$ in a number of works, using a variety of methods, and different expressions for $\mathcal{R}(s,h)$ (for integer or half-integer parameters\footnote{Some of these expressions for $\mathcal{R}(s,h)$ made sense for non-integer values of $s$ as well. However, all of these formulae required the parameter $h$ to be an integer or a half-integer in order to make sense.}) were obtained, see  \cite{Bailey}, \cite{Basor0}, \cite{Conrey}, \cite{Dehaye2}, \cite{combinatorial1}, \cite{Zeta4}, \cite{Winn} for more details.

Recently, by employing a more probabilistic approach in \cite{assiotispainleve} the conjecture was proven for general real values of the exponents $s>-\frac{1}{2}$ and $0\le h <s+\frac{1}{2}$. For these parameter values the main result of \cite{assiotispainleve} reads as follows:
\begin{equation}
\lim\limits_{N\to \infty}\frac{\mathcal{R}_N(s,h)}{N^{s^2+2h}}=\mathcal{R}(s,h) = \frac{G(s+1)^2}{G(2s+1)}2^{-2h}\mathbb{E}\left(\left|\mathbf{X}(s)\right|^{2h}\right). \label{eq:AKWformula}
\end{equation}
Here, $G$ denotes the Barnes G-function, given by
\begin{equation}
    G(1+z) = (2\pi)^\frac{z}{2}\exp\left(-\frac{z+z^2(1+\gamma)}{2}\right)\prod_{j=0}^{\infty}\left(1+\frac{z}{j}\right)^j\exp\left(\frac{z^2}{2j}-z\right)
\end{equation}
where $\gamma$ is the Euler-Mascheroni constant. Our results below will lead to the first explicit evaluation of $\mathcal{R}(s,h)$ for generic real values of $h$ when the parameter $s$ is an integer.

\paragraph{Joint moments of Hardy's function.} As mentioned above, part of the motivation of Hughes in studying the asymptotics of the joint moments $\mathcal{R}_N(s,h)$ was to obtain a precise conjecture for the asymptotics of the joint moments of Hardy's function $ \mathcal{Z} $, defined as follows: 
\begin{equation}
    \mathcal{Z}(y) = \pi^{-iy/2} \frac{\Gamma(1/4 + iy/2)}{|\Gamma(1/4+iy/2)|}\zeta(1/2+iy),
    \label{eq:hardy}
\end{equation}
where $\zeta$ denotes the Riemann zeta function.

Building on the seminal work of Keating and Snaith \cite{Zeta1},\cite{Zeta2}, where the connection between moments of the Riemann zeta function and moments of characteristic polynomials was first understood, Hughes in \cite{Zeta4} conjectured the following, see also \cite{Zeta3}:
\begin{equation}\label{eq:HughesConj}
 \frac{1}{x}\int_0^{x}|\mathcal{Z}(y)|^{2s-2h}\left|\frac{d\mathcal{Z}}{dy}\right|^{2h}dy \sim a(s) \mathcal{R}(s,h) (\log(x))^{s^2+2h},
\end{equation}
as $x \to \infty$\footnote{Here we use the conventional notation $f(x) \sim g(x)$ to denote asymptotic equivalence, i.e. that $\frac{f(x)}{g(x)} \to 1$ as $x \to \infty$.} where the arithmetic factor $a(s)$ is given by:
\begin{equation}
a(s) := \prod_{\text{primes} \ p} \left(1-p^{-1}\right)^{s^2}\sum_{k=0}^{\infty} p^{-k}\left(\frac{\Gamma(k+s)}{\Gamma(k+1)\Gamma(s)}\right)^2. \label{eq:arithmetic}
\end{equation}

The conjecture agrees with rigorous results of Hardy and Littlewood \cite{HardyLittlewood} for $s=1, h=0$, Ingham \cite{Ingham} for $s=2, h=0$ and $s=1, h=1$, Conrey \cite{Conrey2} for $s=2, h=1$ and $s=2, h=2$ and Conrey and Ghosh \cite{Zeta5} for $s=1,h=\frac{1}{2}$. Hughes also stated an analogous conjecture for the joint moments of the Riemann zeta function itself and showed \cite[\S 6.3.]{Zeta4} that for $s,h\in \mathbb{N}$ the two conjectures are equivalent; in particular it is possible to obtain one leading order coefficient from the other, see \cite[6.105]{Zeta4}.

The formula \eqref{eq:AKWformula}, expressing $\mathcal{R}(s,h)$ in terms of the moments of $\mathbf{X}(s)$, thus leads to the following refinement of \eqref{eq:HughesConj}, as given in \cite{assiotispainleve}:
\begin{equation}
\frac{1}{x}\int_0^{x}|\mathcal{Z}(y)|^{2s-2h}\left|\frac{d\mathcal{Z}}{dy}\right|^{2h}dy \sim a(s) \frac{G(s+1)^2}{G(2s+1)} 2^{-2h} \mathbb{E}\left(|\mathbf{X}(s)|^{2h}\right)(\log(x))^{s^2+2h},
\label{eq:conjassiotis}
\end{equation}
valid for $s > - \frac{1}{2}$ and $h \in [0, s+ \frac{1}{2})$. The results of this paper lead to a further refinement of the conjecture above by explicitly evaluating the right-hand side of \eqref{eq:conjassiotis} when the parameter $s$ is an integer.

\paragraph{Ergodic decomposition of the Hua-Pickrell measures.} Lastly, the random variable $\mathbf{X}(s)$ arises naturally in the problem of ergodic decomposition of the Hua-Pickrell measures on the space infinite Hermitian matrices $\mathbb{H}(\infty)$\footnote{This is defined as the projective limit of the spaces of finite Hermitian matrices $\mathbb{H}(N)$ under the natural projections $\mathbb{H}(N+1) \twoheadrightarrow \mathbb{H}(N)$ given by restriction to the top left $N\times N$ submatrix.}. Indeed, this was the first setting \cite{BorodinOlshanskiErgodic}, \cite{Qiu} in which $\mathbf{X}(s)$ appeared. 

The classification of the ergodic measures invariant under the action of $\mathbb{U}(\infty)$\footnote{This is defined as the inductive limit of the finite unitary groups $\mathbb{U}(N)$ under the natural inclusions $\mathbb{U}(N) \hookrightarrow \mathbb{U}(N+1) : \mathbf{U} \mapsto \text{diag}(\mathbf{U},1)$.} by conjugation was derived by Pickrell \cite{Pickrell} and Olshanski and Vershik \cite{OlshanskiVershik}. These measures can be parametrized by the following space of real parameters $\Omega\subset \mathbb{R}^{2\infty+2}$:
\begin{multline}
    \Omega := \big\{\omega= \left(\{\alpha^+_n\}_{n\in \mathbb{N}}, \{\alpha^-_n\}_{n \in \mathbb{N}}, \gamma_1, \gamma_2\right)\in \mathbb{R}^{2\infty+2} : \gamma_2 \ge 0, \ \alpha^{\pm}_n \ge 0 \ \text{and} \ \alpha^{\pm}_n \ge \alpha^{\pm}_{n+1} \ \text{for all} \ n\in \mathbb{N} \\
    \text{and} \sum_{n=1}^{\infty} (\alpha_n^{\pm})^2 < \infty \big\}.
\end{multline}
The parameters of a point $\omega= \left(\{\alpha^+_n\}_{n\in \mathbb{N}}, \{\alpha^-_n\}_{n \in \mathbb{N}}, \gamma_1, \gamma_2\right)\in \Omega$ do have a concrete meaning\footnote{Informally, for $\omega= \left(\{\alpha^+_n\}_{n\in \mathbb{N}}, \{\alpha^-_n\}_{n \in \mathbb{N}}, \gamma_1, \gamma_2\right)\in \Omega$ the $\alpha$ parameters are asymptotic (as $N\to \infty$) normalised eigenvalues, $\gamma_1$ is the asymptotic normalised trace and $\gamma_2$ is closely related to the asymptotic normalised sum of squares of eigenvalues of the $N\times N$ top left submatrix of a random $M_{\omega}$-distributed matrix on $\mathbb{H}(\infty)$ (where $M_{\omega}$ is the ergodic measure parametrised by $\omega\in \Omega$). All these limits exist $M_{\omega}$-almost surely and are deterministic and coincide with the numerical values of the parameters of $\omega$, see \cite{OlshanskiVershik},  \cite{BorodinOlshanskiErgodic} for more details and proofs.}, as explained in the work of Olshanski and Vershik, see \cite{OlshanskiVershik},  \cite{BorodinOlshanskiErgodic} for more details. Moreover, explicit expressions for the characteristic functions for these ergodic measures $M_\omega$, $\omega \in \Omega$, are also known (which uniquely determine them), see for example \cite{OlshanskiVershik} and \cite{BorodinOlshanskiErgodic}.

A distinguished family of probability measures on $\mathbb{H}(\infty)$ were constructed by Borodin and Olshanski in \cite{BorodinOlshanskiErgodic}, depending on a parameter $s \in (-\frac{1}{2}, \infty)$, with the property that when projected onto the top left $N\times N$ submatrix, the classical generalised Cauchy, also known as Hua-Pickrell, ensemble is recovered\footnote{See \cite{HuaPickrell1}, \cite{HuaPickrell2}, \cite{HuaPickrell3}, \cite{HuaPickrell4}, \cite{Qiu}, \cite{BorodinOlshanskiErgodic} and \cite{HuaPickrell5} for information on this ensemble and the related Hua-Pickrell measures; see also \cite{HuaPickrell6}, \cite{HuaPickrell7}, \cite{HuaPickrell8}, \cite{HuaPickrell9}, \cite{HuaPickrell10} for more information on similar measures; see also \cite{HuaPickrellPainleve} for a relation to Painlev\'e trancendents; see also \cite{HuaPickrellStochastic1} and \cite{HuaPickrellStochastic2} for relations to stochastic processes.} and we will say more about this in Section \ref{preliminaries}. Moreover, it was demonstrated in \cite{BorodinOlshanskiErgodic} that for this Hua-Pickrell measure, which we denote here by $\mathfrak{M}^{(s)}$ (and indeed any $\mathbb{U}(\infty)$-invariant probability measure on $\mathbb{H}(\infty)$), there is a unique probability measure $\mu^{(s)}$ on $\Omega$ describing its decomposition into the ergodic measures\footnote{A more precise version of (\ref{ErgodicDecomposition}) is the following. Both $\mathbb{H}(\infty)$ and $\Omega$ are Borel spaces, see \cite{OlshanskiVershik} and \cite{BorodinOlshanskiErgodic} for the details. Then, for any bounded Borel function $\mathfrak{F}$ on $\mathbb{H}(\infty)$ we have:
\begin{align*}
    \mathfrak{M}^{(s)}(\mathfrak{F}) = \int_\Omega \mu^{(s)}(d \omega)M_\omega(\mathfrak{F}),
\end{align*} where $\mathfrak{M}^{(s)}(\mathfrak{F})$ and $M_{\omega}(\mathfrak{F})$ are the integrals of $\mathfrak{F}$ with respect to $\mathfrak{M}^{(s)}$ and the ergodic measure $M_{\omega}$ respectively and for any such $\mathfrak{F}$ the function $\omega\mapsto M_{\omega}(\mathfrak{F})$ is a Borel function on $\Omega$, see in particular Sections 4 and 9 in \cite{BorodinOlshanskiErgodic} for further details and proofs.}:

\begin{equation}\label{ErgodicDecomposition}
\mathfrak{M}^{(s)}(d\mathbf{H}) = \int_\Omega \mu^{(s)}(d \omega)M_\omega(d\mathbf{H}).
\end{equation}

Borodin and Olshanski \cite{BorodinOlshanskiErgodic} were able to explicitly describe the distribution of the parameters $\{\alpha_n^\pm\}_{n \in \mathbb{N}}$ under $\mu^{(s)}$. However, the problem of determining the distribution of $\gamma_1$ and $\gamma_2$ under $\mu^{(s)}$ was unresolved for many years. In an important work, Qiu \cite{Qiu} proved that almost surely $\gamma_2 = 0$ and that $\gamma_1$ is precisely $\mathbf{X}(s)$.

\subsection{Main results}
\paragraph{Painlev\'e equation.} From the discussion above it is evident that understanding the family of random variables $\{\mathbf{X}(s) \}_{s\in (-\frac{1}{2}, \infty)}$ is important. A number of results, including explicit combinatorial formulae for the even moments of these random variables, were proven in \cite{assiotispainleve}. Moreover, one of the results in \cite{assiotispainleve} established a connection between the characteristic function of $\mathbf{X}(s)$ and a special case of the $\sigma$-Painlev\'e III' equation. It was however, for a reason that we explain below, restricted to integer parameters $s$ and the equation was only shown to hold in a small interval around the origin. Our first main result removes both of these restrictions and thus establishes the connection with Painlev\'e for the full range of parameter values $s>-\frac{1}{2}$.
\begin{thm}
Let $s\in \mathbb{R}$ and $s>-\frac{1}{2}$. Define:
\begin{equation}
    \phi^{(s)}(t)=\mathbb{E}\left(e^{\frac{it}{2}\mathbf{X}(s)}\right),
\end{equation}
and the associated function
\begin{equation}
    \tau^{(s)}(t):=t\frac{d}{dt}\log \phi^{(s)}(t).
\end{equation}
Then $\tau^{(s)}(t)$ is $C^\omega$ \footnote{We use the notation $C^\omega$ to denote the space of real analytic functions.} on $\mathbb{R}^* $ and is a solution to a special case of the $\sigma$-Painlev\'e III' equation for $t \in \mathbb{R}^* $:
\begin{equation}
  \left(t\frac{d^2\tau^{(s)}}{dt^2}\right)^2=-4t\left(\frac{d\tau^{(s)}}{dt}\right)^3+\left(4s^2+4\tau^{(s)}\right)\left(\frac{d\tau^{(s)}}{dt}\right)^2+t\frac{d\tau^{(s)}}{dt}-\tau^{(s)}.
  \label{eq:painleve3}
\end{equation}
Moreover, we have the boundary conditions :
\begin{equation}
\begin{cases}
      \tau^{(s)}(0)=0, & \text{for}  \ s>0,\label{eq:boundary1}\\
      & \\
   \frac{d}{dt}\tau^{(s)}(t)\Bigr\rvert_{t = 0}=0, & \text{for} \ s>\frac{1}{2}. 
\end{cases}
\end{equation}
\label{zetainftypainleve}
\end{thm}

\begin{rmk} By using the well-known relation between derivatives of the characteristic function of a random variable and its positive integer moments, along with formula (\ref{eq:AKWformula}), Theorem \ref{zetainftypainleve} readily gives the following expression for $\mathcal{R}(s,h)$, for $h\in \mathbb{N}$ and any $s>h-\frac{1}{2}$ (recall that this restriction is necessary for the joint moments to exist):
\begin{align*}
 \mathcal{R}(s,h) =  (-1)^h\frac{G(s+1)^2}{G(2s+1)}\frac{d^{2h}}{dt^{2h}}\left[\exp\left(\int_0^t\frac{\tau^{(s)}(u)}{u}du\right)\right]\bigg|_{t=0},
\end{align*}
where the function $\tau^{(s)}$ solves the Painlev\'e equation in Theorem \ref{zetainftypainleve} above. This generalises to non-integer parameter $s$ the formula (1.16) of Theorem 2, in \cite{Basor0} which established this result for $s,h\in \mathbb{N}$ (with $s>h-\frac{1}{2}$).

For non-integer $h$ one could in principle use Fourier inversion to obtain the probability density\footnote{This implicitly assumes that $\phi^{(s)}\in L^1(\mathbb{R})$, which we expect to be true in general and in fact show for $s\in (-\frac{1}{2},0]\cup \mathbb{N}$ in the course of proving our results, so that the probability density of $\mathbf{X}(s)$ with respect to the Lebesgue measure exists. However, even without this assumption it is still possible to write down a more involved formula or alternatively one could use one of the numerous known expressions which give the fractional moments of the absolute value of a random variable in terms of its characteristic function; all of which involve some kind of integral operator.} of $\mathbf{X}(s)$ (we note that $\mathbf{X}(s)$ is a symmetric random variable and thus its density is even), then integrate to get the (fractional) moments $\mathbb{E}\left(\left|\mathbf{X}(s)\right|^{2h}\right)$ and thus obtain an integral expression for $\mathcal{R}(s,h)$ in terms of this Painlev\'e transcendent. However, in order to obtain an explicit formula for $\mathcal{R}(s,h)$ one would need as a starting point explicit exact\footnote{Namely, not simply asymptotic expansions.} expressions for $\phi^{(s)}$. These exist for $s\in \mathbb{N}\cup\{0\}$ (the corresponding functions $\tau^{(s)}$ are so-called classical solutions of the $\sigma$-Painlev\'e III' equation, see \cite{NonClassical3}, \cite{NonClassical1}, \cite{NonClassical2} for the precise definition of this terminology) and this is really the essence of Theorems \ref{s0explicit} and \ref{thm:s=0} below. 
\end{rmk}

We prove Theorem \ref{zetainftypainleve} by a limiting procedure, making use of the connection to the finite $N$ Hua-Pickrell measures that we will explain in Section \ref{preliminaries}. This strategy was also employed in \cite{assiotispainleve}, however there a result from \cite{Basor0} was used as input that required $s$ to be an integer\footnote{The second restriction, namely that the equation only holds on an interval is due to the fact that a priori it is not clear whether $\phi^{(s)}$ is non-vanishing on the real line. This requires separate arguments that we present in Section \ref{Proofs}.}. The reason the result from \cite{Basor0} was restricted to integer parameters $s$ was because the proof was using as a starting point a certain formula of Winn \cite{Winn} involving an $s\times s$ determinant with entries given by Laguerre polynomials. 

The starting point of our work is the observation that it is possible to use instead a different formula from \cite{Winn}, that we recall in Proposition \ref{prop:winn} below, valid for all $s>-\frac{1}{2}$, which gives rise to a Hankel determinant of a deformed Laguerre weight. This Hankel determinant formula is well-adapted for an application of the so-called ladder operator method, see \cite{HankeldetLagu}, \cite{Basor1}, which gives as output a Painlev\'e representation. We initially found that this computation was performed in the applied mathematics literature in \cite{Hankeldet} in relation to applications to wireless communications (with some restrictions on the parameters which can be removed by a short analytic continuation argument). However, as pointed out to us by one of the referees, an equivalent result had in fact already been obtained a decade before \cite{Hankeldet} by Forrester and Witte in \cite{ForresterWitte} by a different method that uses Okamoto's $\tau$-function theory of the Painlev\'e equations \cite{Okamoto}. We say a bit more about the history before the statement of Proposition \ref{prop:HankelPainleve1}.

Now, in order to take the $N\to \infty$ limit our arguments are completely different to the ones in \cite{assiotispainleve}. They are complex analytic in nature, while the ones in \cite{assiotispainleve} are mainly based on convergence of moments\footnote{It is important to note that we do not say anything new regarding convergence of the moments in this paper, except for the special case $s=0$ that we solve explicitly in Theorem \ref{s0explicit}.} (and thus closely related to real analysis). Making use of the power of complex analysis allows us to circumvent a number of technical issues that arise in the proof. 

Although the statement of Theorem \ref{zetainftypainleve} is complete, as it covers the full parameter range, it would be very desirable to have an alternative direct proof from the definition of $\mathbf{X}(s)$ using Fredholm determinants.

Finally, using similar arguments we establish in Section \ref{InverseBessel} a connection to another special case of the $\sigma$-Painlev\'e III' equation for the Laplace transform of the sum of inverse points of the Bessel point process.

\paragraph{Explicit expressions for integer $s$.} 
We now obtain explicit expressions for the density and all the complex moments of the absolute value of the random variables $\mathbf{X}(s)$, when the parameter $s$ is an integer. This leads to the first explicit evaluation of $\mathcal{R}(s,h)$ for general real values of $h$, when $s$ is an integer. When on the other hand $h$ is an integer while $s$ is a general real number, then explicit expressions for $\mathcal{R}(s,h)$ already exist in the literature. The case $h=0$ (in particular $\mathbf{X}(s)$ does not appear) is due to the seminal work of Keating and Snaith \cite{Zeta1}. For integer $h\ge 1$ the reader is referred to \cite{assiotispainleve} and the references therein for more details.

We first consider the case $s=0$ which is in some sense exceptional. We show that $\mathbf{X}(0)$ is actually a Cauchy random variable. Moreover, we extend the convergence of the joint moments of the characteristic polynomial $\mathcal{R}_N(0,h)$ to complex $h$ and cover the full range $\Re(h)\in \left(-\frac{1}{2},\frac{1}{2}\right)$. It is a truly remarkable fact that there is no need to take a large $N$ limit but rather we simply have: $N^{-2h}\mathcal{R}_N(0,h)=\mathcal{R}(0,h)$, for all $N\ge 1$.

\begin{thm}
The random variable $\mathbf{X}(0)$ is Cauchy distributed, namely it has probability density with respect to the Lebesgue measure given by:
\begin{align}\label{CauchyDensity}
\frac{1}{\pi(1+x^2)}, \  x \in \mathbb{R}.
\end{align}
Moreover, if $\Re(h)\in (-\frac{1}{2},\frac{1}{2})$, then:
\begin{equation}
     \frac{\mathcal{R}_N(0,h)}{N^{2h}} = \mathcal{R}(0,h) = 2^{-2h}\frac{1}{\cos(\pi h)}, \ \ \text{for \ all} \ N\ge 1.
\end{equation}
\label{s0explicit}
\end{thm}

We now turn our attention to the case $s\in \mathbb{N}$. First, we recall the standard definition of hypergeometric functions:
\begin{equation*}
    _pF_q \left[\begin{matrix}
   a_1,\ldots, a_p \\
    b_1,\ldots, b_q \end{matrix} \ ; \ z\right]=\sum_{k=0}^{\infty} \frac{(a_1)_k\ldots(a_p)_k}{(b_1)_k\ldots(b_q)_k}  \frac{z^k}{k!},
\end{equation*}
where $(a)_j$ denotes the Pochhammer symbol given by $(a)_j:=\prod_{i=1}^{j}(a+i-1)$ and $(a)_0:=1$. 
In the statement and proof of Theorem \ref{thm:s=0}, we will slightly abuse notation by writing, for any $h$ such that the right hand side exists:
\begin{equation}
  \mathcal{R}(s,h) = \frac{G(s+1)^2}{G(2s+1)}2^{-2h}\mathbb{E}\left(\left|\mathbf{X}(s)\right|^{2h}\right).
  \label{eq:Rdefinition}
\end{equation}
 As explained earlier, from the main result of \cite{assiotispainleve}, this expression coincides with the original definition of $\mathcal{R}(s,h)$ as the limit of the rescaled joint moments $\lim_{N\to \infty}N^{-s^2-2h}\mathcal{R}_N(s,h)$ in the range $h\in [0,s+\frac{1}{2})$ \footnote{In particular, in this paper we do not pursue further the problem of convergence of the rescaled joint moments beyond this range, although we do expect it to hold.}.

\begin{thm}
The random variable $\mathbf{X}(1)$ has probability density with respect to the Lebesgue measure given by:
\begin{align}\label{rho1}
    \frac{-1+e^{\frac{2}{1+x^2}}\cos\left(\frac{2x}{1+x^2}\right)}{2\pi}, \  x \in \mathbb{R}
\end{align}
and hence for $\Re(h)\in \left(-\frac{1}{2},\frac{3}{2}\right)$ we have that:
\begin{equation}
    \mathcal{R}(1,h) =2^{-2h} \frac{1}{\cos(\pi h)} 
    {}_1F_1 \left[\begin{matrix}
     -2h \\
    2 \end{matrix} \ ; \ 2\right]. \label{eq:x1moments}
\end{equation}

The random variable $\mathbf{X}(2)$ has probability density with respect to the Lebesgue measure given by:
\begin{equation}\label{rho2}
    \frac{1}{\pi} \times
    \Re\left(\frac{1}{1-ix} {}_2F_2\left[\begin{matrix}
   \frac{5}{2},1 \\ 5,4 \end{matrix} \ ; \ \frac{8}{1-ix}\right]\right), \  x \in \mathbb{R}
\end{equation}
and hence for $\Re(h) \in \left(-\frac{1}{2},\frac{5}{2}\right)$ we have that
\begin{equation}
    \mathcal{R}(2,h) = 2^{-2h}\frac{1}{12\cos(\pi h)} {}_2F_2\left[\begin{matrix}
   \frac{5}{2},-2h \\
    5,4 \end{matrix} \ ; \ 8\right].
    \label{eq:s2moment}
\end{equation}

Finally, for all $s\in \mathbb{N}$, we have the following general expression for the density of $\mathbf{X}(s)$ with respect to the Lebesgue measure:
\begin{multline}
    \rho^{(s)}(x)= (-1)^{s(s-1)/2}\frac{G(2s+1)}{G(s+1)^2}\frac{1}{2\pi}\\
    \times \Re\left(\sum_{k=0}^\infty \left[\sum_{k_1+\ldots+k_s=k}\det\left[\frac{1}{(k_i+i+j-1)!}\right]_{i,j=1,\ldots,s}{k \choose k_1, \dots, k_s}\right] \left(\frac{2}{1-ix}\right)^{k+1}\right), \  x \in \mathbb{R}.
    \label{eq:generalrho}
\end{multline}
Hence, for $s\in \mathbb{N}$, $\Re(h) \in \left(-\frac{1}{2}, s+\frac{1}{2}\right)$ we have:
\begin{multline}
    \mathcal{R}(s,h)=(-1)^{s(s-1)/2}2^{-2h}\frac{1}{\cos(\pi h)} \\
    \times\left(\sum_{k=0}^\infty
    \left[\sum_{k_1+\ldots+k_s=k}\det\left[\frac{1}{(k_i+i+j-1)!}\right]_{i,j=1,\ldots,s}\prod_{j=1}^s \frac{1}{{k_j}!}\right] (-2h)_k 2^k\right).
    \label{eq:momentsgeneral}
\end{multline}
\label{thm:s=0}
\end{thm}
\begin{rmk}
We note that the general expressions \eqref{eq:generalrho} and \eqref{eq:momentsgeneral} readily specialize to the ones for $s=1$, namely \eqref{rho1} and \eqref{eq:x1moments}. To obtain expressions \eqref{rho2} and \eqref{eq:s2moment} for $s=2$ from the general ones we need to take into account a simplification, due to Vandermonde's Identity, that we present in the proof of Theorem \ref{thm:s=0}. We also note that the expression in \eqref{eq:momentsgeneral} is indeterminate for $h \in \mathbb{N}-\frac{1}{2}$ and should be understood as a limit of $h \to m+\frac{1}{2}$ for some $m\in \mathbb{N}$, which can be computed via L'H\^opital's rule (see Remark \ref{rmk:halfint}).
\end{rmk}

An immediate corollary of Theorems \ref{s0explicit} and \ref{thm:s=0} is a refinement of conjecture \eqref{eq:conjassiotis} when $s$ is an integer. We write out the conjecture fully only for $s=0,1,2$ since these are the simplest and most elegant cases. The general form of the conjecture can be obtained from formula \eqref{eq:momentsgeneral}.
\begin{conjecture}
For $h \in [0,\frac{1}{2})$, we have:
\begin{equation}
\frac{1}{x}\int_0^{x}|\mathcal{Z}(y)|^{-2h}\left|\frac{d\mathcal{Z}}{dy}\right|^{2h}dy \sim 2^{-2h}\frac{1}{\cos(\pi h)}(\log(x))^{2h}.    
\end{equation}
For $h \in [0,\frac{3}{2})$, we have:
\begin{equation}
\frac{1}{x}\int_0^{x}|\mathcal{Z}(y)|^{2-2h}\left|\frac{d\mathcal{Z}}{dy}\right|^{2h}dy \sim 2^{-2h} \frac{1}{\cos(\pi h)} 
    {}_1F_1 \left[\begin{matrix}
     -2h \\
    2 \end{matrix} \ ; \ 2\right](\log(x))^{1+2h}.    
\end{equation}
For $h \in [0,\frac{5}{2})$, we have:
\begin{equation}
\frac{1}{x}\int_0^{x}|\mathcal{Z}(y)|^{4-2h}\left|\frac{d\mathcal{Z}}{dy}\right|^{2h}dy \sim \frac{1}{2\pi^2} 2^{-2h}\frac{1}{\cos(\pi h)} {}_2F_2\left[\begin{matrix}
   \frac{5}{2},-2h \\
    5,4 \end{matrix} \ ; \ 8\right](\log(x))^{4+2h}.\end{equation}
\label{conj:s=0}
\end{conjecture}
\paragraph{The sine process.} The sine process, that we denote by $\mathcal{S}$, is the determinantal point process on $\mathbb{R}$ with correlation kernel given by\footnote{Here we have rescaled the process by a factor of $-\pi$, for aesthetic purposes.} (see \cite{Dyson}, \cite{Mehta}):
\begin{equation}
    \mathfrak{K}_{\text{sine}}(x,y)=\frac{\sin(x-y)}{x-y}. \label{eq:sinekernel}
\end{equation}
$\mathcal{S}$ is arguably the most fundamental object in random matrix theory: it is the universal scaling limit of eigenvalues of complex Hermitian random matrices in the bulk of the spectrum; see for example \cite{Johansson}, \cite{Bleher}, and \cite{Deift} for precise statements. It also has close connections to the pair correlations between zeroes of the Riemann zeta function high up the critical line; see for example \cite{Montgomery} and \cite{Bourgade}.

Here, as an immediate corollary of Theorem \ref{s0explicit}, we obtain an alternative proof of the fact that the principal value sum\footnote{The sine process is translation invariant and thus it can be shown that if one simply takes the sum of inverse points without the cutoff then this sum does not converge.} of the inverse points of the sine process is Cauchy distributed. This fact was already established, as a corollary of more general results, by Aizenman and Warzel in \cite{AizenmanWarzel} using spectral theory methods.
\begin{cor}
Let $\mathcal{S}$ denote the sine process on $\mathbb{R}$. Then, the random variable
\begin{equation}
    \lim\limits_{m \to \infty}\left[ \sum_{y \in \mathcal{S}}\frac{1}{ y}\mathbb{1}\left(|y| < m^2 \right)\right]
\end{equation}
is Cauchy distributed, namely with probability density given by \eqref{CauchyDensity}.
\label{cor:sineprocess}
\end{cor}
\paragraph{Acknowledgements.} BB and AS gratefully acknowledge the financial support from the Mathematical Institute, University of Oxford. BB also gratefully acknowledges the financial support from the EPSRC. MAG gratefully acknowledges the financial support from Prof. J.P. Keating's start-up grant. TA is grateful for financial support at the early stages of this work from ERC Advanced Grant  740900 (LogCorRM). We would like to thank the anonymous referees for a number of very useful comments and suggestions which have improved the presentation and we are also particularly grateful to them for pointing out a number of references.

\begin{section}{Proofs of The Main Results}
\subsection{Preliminaries}
\label{preliminaries}
We begin with a number of preliminaries. Let $\mathbb{H}(N)$ denote the linear space of $N \times N$ complex Hermitian matrices. We define, for a parameter $s\in \mathbb{R}$, $s>-\frac{1}{2}$, the Hua-Pickrell measure $\mathfrak{M}_N^{(s)}$ on $\mathbb{H}(N)$ as follows:
\begin{equation}
    \mathfrak{M}_{N}^{(s)}(d\mathbf{H}) := \text{const}\cdot \det\left(\left(1+\mathbf{H}^2\right)^{-s-N}\right) \times d\mathbf{H}, \label{eq:1}
\end{equation}
where $d\mathbf{H}$ is the Lebesgue measure on $\mathbb{H}(N)$ and the constant is chosen so that this is a probability measure on $\mathbb{H}(N)$.

The distribution of the eigenvalues of a random matrix from the ensemble in \eqref{eq:1} is  given by the following probability measure $\mathfrak{m}_N^{(s)}$ on $\mathbb{R}^N/\mathfrak{S}(N)$, where $\mathfrak{S}(N)$ denotes the $N$-th symmetric group, see \cite{BorodinOlshanskiErgodic}:
\begin{align}
     \mathfrak{m}_{N}^{(s)}(d \mathbf{x}) := \frac{1}{\mathfrak{Z}_N^{(s)}}\cdot{\Delta(\mathbf{x})}^2 \prod _{j=1}^{N}\left( 1+x_{j}^{2}\right) ^{-s-N }dx_j,
     \label{eq:huapickrellensemble}
\end{align}
where $\Delta(\mathbf{x})$ is the Vandermonde determinant:
\begin{equation*}
   {\Delta(\mathbf{x})} = \prod _{1\leq l<k\leq N}(x_k-x_l),
\end{equation*}
and the normalization constant is given by:
\begin{equation*}
 \mathfrak{Z}_N^{(s)} = \pi^N2^{-N(N+2s -1)} \cdot \prod_{j=0}^{N-1} \frac{j!\Gamma(2s + N-j)}{\Gamma(s+N-j)^2}.
\end{equation*}
Throughout this paper we denote expectations taken with respect to the measures $\mathfrak{M}_N^{(s)}$ and $\mathfrak{m}_N^{(s)}$ by $\mathbb{E}_N^{(s)}$.

We now make concrete the connection\footnote{In fact this is how the abstract ergodic decomposition results are proven, see \cite{BorodinOlshanskiErgodic}, \cite{Qiu} for more details.} between the Hua-Pickrell measures and the random variable $\mathbf{X}(s)$. If $\mathbf{H}_N$ is a random matrix distributed according to the probability measure $\mathfrak{M}_N^{(s)}$, it was proven by Borodin and Olshanski \cite{BorodinOlshanskiErgodic} that the sequence of random variables $\big\{\frac{1}{N}\text{Tr}(\mathbf{H}_N)\big\}_{N\ge1}$ is convergent in distribution, and by Qiu \cite{Qiu} that the limiting distribution can be identified with that of $\mathbf{X}(s)$, so that we have:
\begin{equation}
   \frac{1}{N}\text{Tr}\left(\mathbf{H}_N\right) \xrightarrow[N\rightarrow\infty]{d} \mathbf{X}(s).
   \label{eq:convergence1}
\end{equation}
Hence, if we consider the characteristic function of the scaled trace of $\mathbf{H}_N$:
\begin{equation}
    \phi_{N}^{(s)}(t) := \mathbb{E}_N^{(s)}\left(e^{\frac{it}{2N}\text{Tr}(\mathbf{H}_N)}\right)
\end{equation}
we note that by \eqref{eq:convergence1} we have that:
\begin{equation}
    \mathbb{E}_{N}^{(s)}\left(e^{\frac{it}{2N}\text{Tr}(\mathbf{H}_N)}\right) \xrightarrow{N\rightarrow\infty} \mathbb{E}\left(e^{\frac{it}{2}\mathbf{X}(s)}\right),
    \label{eq:convergence2}
\end{equation}
uniformly on compact subsets of $\mathbb{R}$. Finally, we define:
\begin{equation}
    \tau_{N}^{(s)}(t) := t \frac{d}{dt} \log(\phi_{N}^{(s)}(t)).
\end{equation}

\subsection{Proofs}
\label{Proofs}
\begin{prop}
Let $s\in \mathbb{R}$ and $s>-\frac{1}{2}$. Then, for $ t\in \mathbb{R}^*$ , $\tau_{N}^{(s)}(t)$ is a solution to a particular Painlev\'e V equation:
\begin{multline}
    \left(t\frac{d^2\tau_{N}^{(s)}}{dt^2}\right)^2=-4t\left(\frac{d\tau_{N}^{(s)}}{dt}\right)^3+\left(4s^2+4\tau_{N}^{(s)}+\frac{t^2}{N^2}\right)\left(\frac{d\tau_{N}^{(s)}}{dt}\right)^2\\+t\left(1+\frac{2s}{N}-\frac{2\tau_{N}^{(s)}}{N^2}\right)\frac{d\tau_{N}^{(s)}}{dt}-\left(1+\frac{2s}{N}-\frac{\tau_{N}^{(s)}}{N^2}\right)\tau_{N}^{(s)}.
    \label{eq:painleve5}
\end{multline}
\label{zeta_Npainleve}
\end{prop}
Our starting point is the following remarkable integral identity due to Winn from \cite{Winn}, see Proposition 3 therein. For the convenience of the reader and completeness of the paper we outline Winn's proof from \cite{Winn}.
\begin{prop}[B. Winn]
Let $s\in \mathbb{C}$ with $\Re(s)>-\frac{1}{2}$ and $t>0$. Then,
\begin{multline}
\int_{-\infty}^{\infty}\ldots \int_{-\infty}^{\infty} \prod_{j=1}^N
\frac{e^{itx_j}}{\left(1+{x_j}^2\right)^{s+N}}  {\Delta(\mathbf{x})}^2 d\mathbf{x} \\= \frac{\pi^N}{2^{(N+2s-1)N}} \prod_{j=0}^{N-1}\frac{1}{\Gamma\left(s+1+j\right)^2} \cdot e^{-Nt}\int_{0}^{\infty}\ldots \int_{0}^{\infty} \prod_{j=1}^N \left(y_j+2t\right)^s{y_j}^s e^{-y_j}{\Delta(\mathbf{y})}^2 d\mathbf{y}.
\label{eq:winn0}
\end{multline}
\label{prop:winn}
\end{prop}
\begin{proof}
 It suffices to prove the equality up to a constant. The constant can then be obtained by taking the limit $t\to 0$, since both sides then have explicit evaluations using the Selberg integral, see \cite{Selberg}.
 
 Using the homogeneity of $\Delta(\mathbf{x})$, rewrite the integral on the left-hand side of \eqref{eq:winn0} as\footnote{Here we use the notation $f(\mathbf{x}) = \left(f(x_1),\dots,f(x_n)\right)$ to denote the evaluation of a scalar function $f:\mathbb{R}\to \mathbb{R}$ at a vector argument, e.g. $\frac{1}{\mathbf{x}} = (\frac{1}{x_1},\dots,\frac{1}{x_n})$.}:
\begin{multline}
\int_{-\infty}^{\infty}\ldots \int_{-\infty}^{\infty} \prod_{j=1}^N
\frac{e^{itx_j}}{\left(1+{x_j}^2\right)^{s+1}}  {\Delta\left(\frac{1}{1+i\mathbf{x}}\right)\Delta\left(\frac{1}{1-i\mathbf{x}}\right)} d\mathbf{x} \\ =
N! \det \left[ \int_{-\infty}^\infty  \frac{e^{itx} dx}{(1+ix)^{s+1+k}(1-ix)^{s+1+j}}\right]_{j,k = 0, \dots,N-1.},
\label{eq:winn1}
\end{multline}
where in the last line we have used the Andr\'eief identity. It was then shown in \cite[4.15]{Winn} that we have the following relation between one dimensional integrals:
\begin{equation}
    \int_{-\infty}^\infty  \frac{e^{itx} dx}{(1+ix)^{s+1+k}(1-ix)^{s+1+j}} = 
    Ce^{-t} \int_0^\infty y^{s+j}\left(y+2t\right)^{s+k}e^{-y}dy,
\label{eq:winn2}
\end{equation}
for a constant $C$ independent of $j,k$ and $t$. Therefore, combining the last line of \eqref{eq:winn1} with the right-hand side of \eqref{eq:winn2}, and using the Andr\'eief identity once again, noting that $\Delta\left(\mathbf{y}+2t\right) = \Delta(\mathbf{y})$, we obtain the formula \eqref{eq:winn0}, up to a constant.
\end{proof}
\begin{proof}[Proof of Proposition \ref{zeta_Npainleve}]
By definition, we have that:
\begin{equation}
     \phi_{N}^{(s)}(t) = \mathbb{E}_{N}^{(s)}\left(e^{\frac{it}{2}\sum_{i=1}^N \frac{x_i}{N}}\right)=\frac{1}{N! \mathfrak{Z}_N^{(s)}} \int_{-\infty}^{\infty}\ldots \int_{-\infty}^{\infty} \prod_{j=1}^N
\frac{e^{\frac{it}{2N}x_j}}{\left(1+{x_j}^2\right)^{s+N}}  {\Delta(\mathbf{x})}^2 d\mathbf{x}.
\label{eq:phiintegral}
\end{equation}
Now making use of the integral identity \eqref{eq:winn0}, we obtain for $t\ge0$:
\begin{equation}
   \phi_{N}^{(s)}(t)=\frac{1}{C_N^{(s)}}\cdot e^{-t/2}\int_{0}^{\infty}\ldots \int_{0}^{\infty} \prod_{j=1}^N \left(y_j+\frac{t}{N}\right)^s{y_j}^s e^{-y_j}{\Delta(\mathbf{y})}^2 d\mathbf{y}
   \label{eq:laguerre}
\end{equation}
where $C_N^{(s)}$ is a normalization constant, given by:
\begin{equation}
    C_N^{(s)} = N! \prod_{j=1}^{N}\Gamma(j)\Gamma(2s+j).
    \label{eq:constant}
\end{equation}
 An application of the Andr\'eief identity to \eqref{eq:laguerre} yields:
\begin{equation*}
  \phi_{N}^{(s)}(t)= \frac{N!}{C_N^{(s)}}\cdot e^{-t/2} \det\left[\int_{0}^{\infty}y^{j+k}\left(y+\frac{t}{N}\right)^s{y}^se^{-y}dy\right]_{j,k=0,\ldots, N-1},
\end{equation*}
so that 
\begin{equation}
 \tau_{N}^{(s)}(t)=-\frac{t}{2}+t\frac{d}{dt}\log  \det\left[\int_{0}^{\infty}y^{j+k}\left(y+\frac{t}{N}\right)^s{y}^se^{-y}dy
 \right]_{j,k=0,\ldots, N-1}.
\end{equation}
Note that $\phi_N^{(s)}$ and hence $\tau_N^{(s)}(t)$ are even functions, which can be seen by the change of variables $x_j \mapsto -x_j$ in \eqref{eq:phiintegral}. Taking this into account, a simple calculation reveals that if $\tau_N^{(s)}(t)$ satisfies the Painlev\'e equation \eqref{eq:painleve5} for $t>0$, then it is also a solution for $t<0$. Thus, we may restrict to $t>0$, and the result follows immediately as a corollary of the next proposition.
\end{proof}

The following proposition, in an equivalent form, is originally due to Forrester and Witte, see Proposition 3.2 in \cite{ForresterWitte}, using Okamoto's $\tau$-function theory of the Painlev\'e equations \cite{Okamoto}. A decade later this result was proven again\footnote{In fact, in the intervening years the result was also studied in the physics literature, see \cite{Osipov}, by yet another method.} using the ladder operator method (see for example \cite{Basor1}, \cite{Basor4} for more on this technique) in \cite{Hankeldet}, but with the restriction $\alpha>0$. The restriction to $\alpha>0$ is due to the fact that the proofs of certain intermediate results in \cite{Hankeldet} require a number of integrations by parts which are no longer valid in the range $-1<\alpha\le 0$. Here, instead of simply citing the equivalent proposition from \cite{ForresterWitte}, we give a short analytic continuation argument that extends the result of \cite{Hankeldet} to the full range of parameters mainly because analogous arguments can also be used to  extend the intermediate results in \cite{Hankeldet}, on certain orthogonal polynomials, which are of independent interest.

\begin{prop}
Let 
\begin{equation}
    F_N(t;\alpha)=\det\left[\int_{0}^{\infty}y^{j+k}w(y;t,\alpha)dy
 \right]_{j,k=0,\ldots, N-1}
\end{equation}
and the associated function
\begin{equation}
    H_N(t;\alpha)=t\frac{d}{dt}\log F_N(t;\alpha)
\end{equation}
where
\begin{equation}
    w(y;t ,\alpha)=\left(y+t\right)^\lambda {y}^\alpha e^{-y}.
\end{equation}
Then, we have that for $\alpha > -1, t > 0$ and $\lambda \in \mathbb{R}$:
\begin{multline}
    \left(t\frac{d^2H_N}{dt^2}\right)^2=\left(t\frac{dH_N}{dt}-H_N+\frac{dH_N}{dt}(2N+\alpha+\lambda)+N\lambda\right)^2\\-4\frac{dH_N}{dt}\left(t\frac{dH_N}{dt}-H_N+N(N+\alpha+\lambda)\right)\left(\frac{dH_N}{dt}+\lambda\right).
    \label{eq:hankelchen}
\end{multline}
\label{prop:HankelPainleve1}
\end{prop}
\begin{proof}
As mentioned above, in \cite{Hankeldet} the result was proven for $\alpha>0$ and here we simply extend the validity of \eqref{eq:hankelchen} to $\alpha>-1$ by proving that $H_N,\frac{d}{dt}H_N$ and $\frac{d^2}{dt^2}H_N$ can be extended as functions of $\alpha$, for fixed $t>0$ and $\lambda \in \mathbb{R}$, analytically to a neighbourhood $V_N \subseteq \mathbb{C}$ of the part of the real line $(-1,\infty)$.
For the remainder of this proof, let $t>0$ and $\lambda\in \mathbb{R}$ be fixed and arbitrary. Firstly, note that the function
\begin{equation}
    \beta \mapsto \int_{0}^{\infty}y^{\beta}(y+t)^{\lambda}e^{-y}dy
\end{equation}
is holomorphic on $\Re(\beta)>-1$, which can be seen by combining Fubini's and Morera's theorems. Hence, by using rules for derivatives of determinants and differentiation under the integral sign we get that $\alpha\mapsto\frac{d^p}{dt^p}F_N(t;\alpha)$ is holomorphic for $p=0,1,2,3, \ldots$.
Now, by rewriting $F_N(t;\alpha)$ as in the integral in \eqref{eq:laguerre} via the Andr\'eief identity, we see that
\begin{equation}
    F_N(t;\alpha)>0, \ \ \text{for all} \ \ \alpha \in (-1,\infty),
\end{equation}
 so that by continuity of $\alpha \mapsto F_N(t;\alpha)$, there exists $V_N$ such that  $(-1,\infty) \subseteq V_N \subseteq \mathbb{C}$ such that $\left|F_N(t;\alpha)\right|>0$, for all $ \alpha \in V_N$. Hence, the left hand side and right hand side of \eqref{eq:hankelchen} are two analytic functions on $V_N$ that agree on $(0,\infty)$ so that they must agree on the whole of $V_N$. The proof for the case $\alpha>-1$ is now complete.
\end{proof}
In order to prove Theorem \ref{zetainftypainleve} we will need the following proposition. It is important to note that the use of the obvious candidates to perform these analytic extensions in the proposition below, namely the expressions as characteristic functions, will not work. These expressions do not exist off the real line due to the fact that the random variables involved only have a finite number of integer moments.
\begin{prop}[Analytic continuation]
Let $s\in \mathbb{R}$ and $s>-\frac{1}{2}$. Then, there exist holomorphic functions  $f_N$ for $N=1,2,\dots$ and $f$ on $\{ z \in \mathbb{C}: \Re z  > 0 \}$, with $f_N(0) = f(0)=1$, such that
\begin{align}
    \phi_N^{(s)} = \left.f_N\right|_{[0,\infty)} && \text{and} &&  \phi^{(s)} = \left.f\right|_{[0,\infty)}.
\end{align}
Moreover, for $\Re z > 0, \ p=1,2,\dots$ we have:
\begin{equation}
\lim\limits_{N\to\infty} \frac{d^p}{dz^p}f_{N}(z) = \frac{d^p}{dz^p}f(z).
\end{equation}
\label{prop:analytic}
\end{prop}
\begin{proof}[Proof of Proposition \ref{prop:analytic}]
We use an argument based on Montel's theorem and Morera's theorem, see  \cite{SteinShakarchi}. For $\{z \in \mathbb{C}: \Re z > 0\}$, we define as in \eqref{eq:laguerre}:
\begin{equation}
    f_N(z) := \frac{1}{C_N^{(s)}}\cdot\int_{0}^{\infty}\ldots \int_{0}^{\infty} e^{-z/2} \prod_{j=1}^N \left(y_j+\frac{z}{N}\right)^s{y_j}^s e^{-y_j}{\Delta(\mathbf{y})}^2 d\mathbf{y}.
\label{eq:analyticf}
\end{equation}
Then, as in \eqref{eq:laguerre} we have that $\phi_N(t) = f_N(t)$ for $t \in [0, \infty)$. We claim that $f_N(z)$ is holomorphic on $\{z \in \mathbb{C}: \Re z > 0\}$. Towards this end, we note that the integrand in \eqref{eq:analyticf} is bounded by an integrable function on any compact subset of $\{z \in \mathbb{C}: \Re z > 0 \}$, and hence by Fubini's theorem, for any closed path $\gamma$ contained in $\{z \in \mathbb{C}: \Re z > 0 \}$ we have that
\begin{equation}
    \int_\gamma f_N(z) dz = \frac{1}{C_N^{(s)}}\cdot \int_{0}^{\infty}\ldots \int_{0}^{\infty} \int_\gamma e^{-z/2} \prod_{j=1}^N \left(y_j+\frac{z}{N}\right)^s{y_j}^s e^{-y_j}{\Delta(\mathbf{y})}^2 dz d\mathbf{y} = 0,
\end{equation}
where we have used Cauchy's theorem. Hence, by Morera's theorem $f_N(z)$ is holomorphic on $\{z\in \mathbb{C} : \Re z > 0 \}$. We now claim further that the family $\left\{f_N\right\}_{N\ge 1}$ is uniformly bounded on compact subsets of $\{z \in \mathbb{C}: \Re z > 0\}$. For, if $K$ is such a compact subset there is a constant $M_K > 0$ such that
\begin{equation*}
    \left|y+\frac{z}{N}\right| \le \left|y+\frac{M_K}{N}\right|,
\end{equation*}
uniformly for $z \in K$, $y > 0$ and $N \ge 1$. Therefore, in the case $s>0$, it follows from the formula \eqref{eq:analyticf} that 
\begin{equation*}
    |f_N(z)| \ \le \ |e^{-\frac{z}{2}}|e^{\frac{M_K}{2}}\left|\phi_N^{(s)}(M_K)\right|  \ \le \ e^\frac{M_K}{2},
\end{equation*}
uniformly for $z \in K$ and $N \ge 1$, where in the last inequality we have used the fact that $\phi_N^{(s)}$ is a characteristic function \eqref{eq:phiintegral}. In the case $s \in (-\frac{1}{2},0]$, we note that $|f_N(z)| \le 1$ whenever $\Re z > 0$, for all $N\ge 1$. Hence, by Montel's theorem the family $\left\{f_N\right\}_{N\ge 1}$ is normal, and so every subsequence has a sub-subsequence converging uniformly on compacts of $\{z\in \mathbb{C} : \Re z > 0 \}$. As all these limits are holomorphic and agree on $(0,\infty)$ by \eqref{eq:convergence2}, they are equal. This property implies that the family $\left\{f_N\right\}_{N\ge 1}$ is convergent uniformly on compacts to a holomorphic function which we denote by $f(z)$. By properties of uniform limits of holomorphic functions, the sequences of derivatives $\left\{\frac{d^p}{dz^p}f_{N}(z)\right\}_{N\ge 1}$, $p = 0,1,2,\dots$, also converge (uniformly on compacts) to $\frac{d^p}{dz^p}f(z)$. Thus $f(z)$ is the required analytic continuation to $\{z \in \mathbb{C}: \Re z > 0 \}$, and we have the required convergence of derivatives. 
\end{proof}

\begin{proof}[Proof of Theorem \ref{zetainftypainleve}]
We first show that the characteristic functions $\phi_N^{(s)}(t)$ and $\phi^{(s)}(t)$ are strictly positive for $s>-\frac{1}{2}$. 

For $s>0$, using the integral representation of $\phi_{N}^{(s)}(t)$ in \eqref{eq:laguerre} we can see that $\phi_{N}^{(s)}(t)e^{\frac{t}{2}}$ is increasing in $t$ for $t\geq0$. Hence, we get that for all $ N\geq1$ and $t>0$,
\begin{align*}
   \phi_{N}^{(s)}(t)\geq e^{-\frac{t}{2}}\phi_{N}^{(s)}(0)=e^{-\frac{t}{2}}
\end{align*}
and hence, as $\phi_{N}^{(s)}(-t)=\phi_{N}^{(s)}(t)$, which can be seen by the change of variables $x_j \mapsto -x_j$ in \eqref{eq:phiintegral}, we have that for all $t \in \mathbb{R}$:
\begin{align}
   \phi_{N}^{(s)}(t) \geq \exp\left(\frac{-|t|}{2}\right)>0.
   \label{eq:phiinequality}
\end{align}
By \eqref{eq:convergence2}, we see that \eqref{eq:phiinequality} implies that $\phi^{(s)}(t)$ is non-vanishing for $s>0, t \in \mathbb{R}$.
For $s \in (-\frac{1}{2},0]$ we need to argue more indirectly. By the formula \eqref{eq:laguerre} we see that $\phi_N^{(s)}(t) \ge 0$ for all $t \in \mathbb{R}$ and $N\ge 1$, and hence $\phi^{(s)}(t) \ge 0$ for all $t \in \mathbb{R}$. Moreover, for $s \in (-\frac{1}{2},0]$, by \eqref{eq:laguerre}, for all $N\ge 1$, $\phi_N^{(s)}(t)$, and thus also $\phi^{(s)}(t)$, are non-increasing on $(0, \infty)$. Hence, if these functions vanish at some $r>0$ then they are identically $0$ on $[r, \infty)$. However, since by Proposition \ref{prop:analytic} they are restrictions of holomorphic functions, this would imply that they are identically $0$ for $t>0$, which is a contradiction since a characteristic function is non-vanishing on a real neighbourhood around zero. Hence, $\phi_N^{(s)}(t)$ and $\phi^{(s)}(t)$ are strictly positive for $t\in(0,\infty)$ and the result follows from noting as before that these functions are even.

This implies that $\tau_N^{(s)}(t)$ and $\tau^{(s)}(t)$ are well-defined for all $N\ge 1, s>-\frac{1}{2}$, and $t \in \mathbb{R}$. Hence, by Proposition \ref{prop:analytic}, and using that the functions $\tau_N^{(s)}$ are even,  we know that for all $t\in \mathbb{R}^*$ and $p=0,1,2,\dots$:
\begin{equation*}
    \frac{d^p}{dt^p}\tau_{N}^{(s)}(t) \xrightarrow{N\rightarrow\infty} \frac{d^p}{dt^p}\tau^{(s)}(t).
\end{equation*}
Hence the Painlev\'e equation \eqref{eq:painleve3} now follows immediately by taking the limit $N \to \infty$ in the equation \eqref{eq:painleve5}. 

We now prove $\tau^{(s)}$ is real-analytic on $\mathbb{R}^*$. Let $f$ denote the analytic continuation of $\phi^{(s)}$ to $\{ z : \Re z >0 \}$, as in Proposition \ref{prop:analytic}. We know from the above that for $t\in (0,\infty)$, 
\begin{equation*}
\Re\left(f(t)\right) = \Re\left(\phi^{(s)}(t)\right) = \phi^{(s)}(t) > 0.
\end{equation*}
By continuity this implies that $\Re\left(f(t)\right) > 0$ on an open set $V$ with $(0,\infty) \subseteq V \subseteq \{ z : \Re z > 0\}$. Therefore we can define a branch of $\log(z)$ such that $\log(f(z))$ is holomorphic on $V$. As $\tau^{(s)}$ is even, this implies that it is real-analytic on $\mathbb{R}^*$.
 
To establish the boundary conditions for $s>\frac{1}{2}$ we have to employ a different method. Whenever $r<2s+1$, we note that $\mathbb{E}\left(\left|\sum_{i=1}^N {x_i}\right|^r\right)<\infty$ and that the sequence $\left\{\left|\sum_{i=1}^N \frac{x_i}{N}\right|^r\right\}_{N \ge 1}$ is uniformly integrable (see \cite[Proposition 2.10]{assiotispainleve}, and also Proposition \ref{prop:uniformInt} below, where we prove a similar statement using the same idea, and hence, also making use of \eqref{eq:convergence2}:
\begin{equation*}
    \frac{d^p}{dt^p}\phi_{N}^{(s)}(t) \xrightarrow{N\rightarrow\infty} \frac{d^p}{dt^p}\phi^{(s)}(t),
\end{equation*}
for any $t \in \mathbb{R}$ and $p\in \{0,1\}$ whenever $s>0$, and $p \in \{0,1,2\}$ whenever $s>\frac{1}{2}$. Hence, we conclude, again using that the $\phi_N^{(s)}$ and $\phi^{(s)}$ are non-vanishing, that:
\begin{equation}
    \frac{d^p}{dt^p}\tau_{N}^{(s)}(t) \xrightarrow{N\rightarrow\infty} \frac{d^p}{dt^p}\tau^{(s)}(t),\label{eq:boundaryconvergence}
\end{equation}
for $t \in \mathbb{R}$ and $p=0$ whenever $s>0$, and $p \in \{0,1\}$ whenever $s>\frac{1}{2}$. The boundary conditions $\tau_N^{(s)}(0) = 0$ for $s>0$ and $\left.\frac{d}{dt}\tau_N^{(s)}(t)\right|_{t=0}=0$ for $s>\frac{1}{2}$ are computed in \cite{assiotispainleve}. Hence by \eqref{eq:boundaryconvergence} we deduce the boundary conditions \eqref{eq:boundary1}.
\end{proof}
\label{zetainf}

\begin{rmk}
The simple observation made earlier that for $s\le 0$, the integral representation of $\phi_{N}^{(s)}(t)\cdot e^{t/2}$ given in equation \eqref{eq:laguerre} is non-increasing in $t$ gives the following bound: $\phi^{(s)}(t) \le e^{-t/2}$. By Fourier inversion, since the characteristic function of $\mathbf{X}(s)$ is in $L^1(\mathbb{R})$, this readily implies the non-trivial result that the law of $\mathbf{X}(s)$ has a bounded and continuous density with respect to the Lebesgue measure. In fact, due to the exponential decay of the characteristic function, the density is $C^\infty$-smooth. We expect this result to be true for $s>0$ as well\footnote{Clearly, for $s\in \mathbb{N}$ we already have explicit expressions for the density of $\mathbf{X}(s)$ from Theorem \ref{thm:s=0}.}. This is likely to require a more elaborate argument and we do not pursue it further in this paper.
\end{rmk}

We now prove Theorem \ref{s0explicit}. It is worth noting that from the definition of the characteristic function $\phi_N^{(s)}$ it is unclear whether any value of the parameter $s$ is special, while if one looks at formula \eqref{eq:laguerre} it becomes evident that $s=0$ is exceptional. We believe this remarkable observation was missed in the literature due to the fact that formulae such as \eqref{eq:winn0} were not thought of in probabilistic terms, but rather as simply some intermediate formulae required to prove the determinantal representation in terms of Laguerre polynomials\footnote{In fact, we will also make use of this expression in the proof of Proposition \ref{PropExplicitChar} below.} mentioned in the introduction.

\begin{proof}[Proof of  Theorem \ref{s0explicit}.]
 When $s=0$, the integral in the formula \eqref{eq:laguerre} evaluates simply to $C_N^{(0)}$, as in \eqref{eq:constant}. Therefore, it follows that $\phi_N^{(0)}(t) = e^{-|t|/2}$ for all $N = 1, 2, \dots $. Thus, by \eqref{eq:convergence2} we get that $\mathbb{E}(e^{it \mathbf{X}(0)/2}) = e^{-|t|/2}$, i.e. $\mathbf{X}(0)$ is Cauchy distributed.
 
 For the second part, note that by \cite[Proposition 2.7]{assiotispainleve} we have, for $\Re(h)\in (-\frac{1}{2},\frac{1}{2})$ \footnote{The result in \cite[Proposition 2.7]{assiotispainleve} is stated for real $h$ but the argument goes through verbatim for complex $h$ as well. See also \cite[Proposition 2]{Winn}.} that:
\begin{equation}
   \frac{  \mathcal{R}_N(0,h)}{N^{2h}}=2^{-2h}\mathbb{E}_{N}^{(s)}\left(\left|\sum_{i=1}^N\frac{x_i}{N}\right|^{2h}\right).
\end{equation}
Now, for $s=0$, for all $N\ge 1$, $\sum_{i=1}^N\frac{x_i}{N}$ and $\mathbf{X}(s)$ are identically distributed. Thus, for all $ N \geq 1$, $\Re(h) \in (-\frac{1}{2},\frac{1}{2})$ we have:
\begin{equation}
    \frac{\mathcal{R}_N(0,h)}{N^{2h}}=2^{-2h}\mathbb{E}_{N}^{(0)}\left(\left|\sum_{i=1}^N\frac{x_i}{N}\right|^{2h}\right)=\mathcal{R}(0,h)
    =2^{-2h} \mathbb{E}\left(\left|\mathbf{X}(0)\right|^{2h}\right)=\int_{0}^{\infty}\frac{x^{2h}}{\pi \left(x^2+1\right)} dx
    \label{eq:s0proof}
\end{equation}
From \cite[3.241.2]{IntegralsBook}, we have that for $\Re(h)\in(-\frac{1}{2},\frac{1}{2})$:
\begin{equation*}
    \int_{0}^{\infty}\frac{x^{2h}}{\pi \left(x^2+1\right)} dx=\frac{1}{2\cos(\pi h)}.
\end{equation*}
Substituting this into the right-hand side of \eqref{eq:s0proof} gives the desired result.
\end{proof}
Before proving Theorem \ref{thm:s=0}, we need the following proposition that we essentially extract from the results of \cite{Forrester2}, \cite{Basor0}. At the end of this section we also present a short elementary proof of this result for $s=1$.
\begin{prop}\label{PropExplicitChar}
Let $s\in \mathbb{N}$. Then, $\phi^{(s)}(t)$ is given explicitly  as follows:
\begin{equation}
    \phi^{(s)}(t)= (-1)^{s(s-1)/2}\frac{G(2s+1)}{G(s+1)^2} \times\frac{\det\left[I_{j+k+1}\left(2\sqrt{|t|}\right)\right]_{j,k=0,1,\ldots s-1}}{e^{|t|/2}|t|^{s^2/2}},
    \label{explicitphi}
\end{equation}
where $I_\alpha$ denotes the modified Bessel function of the first kind and $G$ denotes the Barnes $G$-function.
\begin{proof}
By \cite[Proposition 4.5.]{Winn} for $s\in\mathbb{N}$ we have that:
\begin{equation}
    \phi_N^{(s)}(t)= (-1)^{s(s-1)/2}\prod_{j=0}^{N-1} \frac{\Gamma(s+N-j)^2}{j!\Gamma(2s + N-j)} e^{-|t|/2} \det\left[L_{N+s-1-i-j}^{2s-1}\left(-\frac{|t|}{N}\right)\right]_{i,j=0,\ldots,s-1}\label{LaguerreDeterminantExpression}
\end{equation}
where $L_n^{(\alpha)}(x)$ denotes the Laguerre polynomial of order $n$ and parameter $\alpha$ (see \cite{Winn}).

The large $N$ limit of the logarithmic derivative of the right hand side of \eqref{LaguerreDeterminantExpression} was first established in \cite{Forrester2} and also using Riemann-Hilbert problem methods in \cite{Basor0}. Thus, using for example \cite[eq. 5-79]{Basor0} we have, in our notation, that\footnote{There is a typo in \cite[eq. 5-79]{Basor0}, namely a missing factor of $\frac{1}{2}$ for $s^2/t$, which has been corrected here.} for $s\in\mathbb{N}$:
\begin{equation}
    \frac{d}{dt} \log \phi_N^{(s)}(t) \xrightarrow{N \rightarrow \infty} \frac{d}{dt} \log \left(\frac{\det\left[I_{j+k+1}\left(2\sqrt{|t|}\right)\right]_{j,k=0,1,\ldots s-1}}{e^{|t|/2}|t|^{s^2/2}} \right).
\end{equation}
Noting that $\frac{d}{dt} \log \phi_N^{(s)}(t)\xrightarrow{N\rightarrow \infty} \frac{d}{dt} \log \phi^{(s)}(t)$ for $s>0$ by using the results in the proof of Theorem \ref{zetainftypainleve} we obtain equality \eqref{explicitphi} up to a multiplicative constant. To recover this constant we observe that both sides of \eqref{explicitphi} must equal $1$ at $t=0$ and we note that the evaluation of the right hand side of \eqref{explicitphi} at $t=0$ can be obtained by taking $h=0$ in \cite[Corollary 1.5]{assiotispainleve}.
\end{proof}
\end{prop}
 
\begin{proof}[Proof of  Theorem \ref{thm:s=0}.]
We first seek to recover the density of $\mathbf{X}(s)$, where here and for the rest of the proof we assume $s\in \mathbb{N}$. Observe that using the bound (see \cite{Luke}): 
\begin{equation}
    I_n(t) \leq \frac{t^n}{2^n n!}e^t 
    \label{eq:besselbound}
\end{equation}
we see, by expanding the determinant in \eqref{explicitphi} as a sum over $\mathfrak{S}(s)$, that $\phi^{(s)}$ is in $L^1(\mathbb{R})$.
Therefore indeed, by Fourier inversion, we can obtain an expression for the density function $\rho^{(s)}(x)$ of $\mathbf{X}(s)$. Namely, we have that for $x \in \mathbb{R}$:
\begin{align*}
\rho^{(s)}(x) & = \frac{1}{2\pi}\int_{-\infty}^{\infty}e^{ixt}\phi^{(s)}(2t) dt \\ & = \frac{1}{\pi}\Re\left(\int_{0}^{\infty}e^{ixt}\phi^{(s)}(2t) dt \right),
\end{align*}
using the fact that the characteristic function $\phi^{(s)}(2t)$ is even. Now, using the expression \eqref{explicitphi}, we get:
\begin{equation*}
    \rho^{(s)}(x) =  \mathsf{V}^{(s)}\frac{1}{\pi}\Re\left(\int_{0}^{\infty}e^{(ix-1)t}\times\frac{\det\left[I_{j+k+1}\left(2\sqrt{2t}\right)\right]_{j,k=0,1,\ldots s-1}}{(2t)^{s^2/2}} dt\right),
\end{equation*}
where for brevity here and for the rest of the proof we will write:
\begin{equation*}
    \mathsf{V}^{(s)}=(-1)^{s(s-1)/2}\frac{G(2s+1)}{G(s+1)^2}.
\end{equation*}
Hence, expanding the determinant and using the substitution $t\mapsto t^2$ gives:
\begin{equation}
    \rho^{(s)}(x) =  \mathsf{V}^{(s)}\frac{1}{\pi}\Re\left(\sum_{\sigma \in \mathfrak{S}(s)}\sgn(\sigma)\int_{0}^{\infty}e^{(ix-1)t^2}\times\frac{\prod_{j=1}^s\left[I_{j+\sigma(j)-1}\left(2\sqrt{2}t\right)\right]}{2^{s^2/2-1}t^{s^2-1}} dt\right).
    \label{Besselprodintegral}
\end{equation}
Noting that modified Bessel functions of the first kind have the following expansion:
\begin{equation}
    I_\alpha(x)=\sum_{k=0}^\infty \frac{x^{2k+\alpha}}{k! \Gamma(k+\alpha+1)2^{2k+\alpha}},
\end{equation}
we can expand a finite product of modified Bessel functions of the first kind with integer parameters as follows:
\begin{equation*}
    \prod_{j=1}^s I_{\nu_j}(2t) = \sum_{k=0}^{\infty}\left[\sum_{k_1+\dots+k_s = k}\prod_{j=1}^s\frac{1}{k_j!(k_j+\nu_j)!}\right]t^{2k+\nu_1+\dots+\nu_s}.
\end{equation*}
Using this, we can simplify the expression for the density as
\begin{multline*}
    \rho^{(s)}(x) =  \mathsf{V}^{(s)}\frac{1}{\pi}\sum_{\sigma \in \mathfrak{S}(s)}\sgn(\sigma)\\
    \times \sum_{k=0}^{\infty}\left[\sum_{k_1+\dots+k_s = k}\prod_{j=1}^s\frac{1}{k_j!(k_j+j+\sigma(j)-1)!}\right]\Re\left(\int_{0}^{\infty}e^{(ix-1)t^2}\times\ 2^{k+1}t^{2k+1} dt\right),
\end{multline*}
because interchanging the sum and integral is justified as we explain next. Indeed, for each fixed permutation $\sigma\in \mathfrak{S}(s)$, we have, for all $x\in \mathbb{R}$, that:
\begin{equation}
    \int_0^\infty \sum_{k=0}^{\infty}\left[\sum_{k_1+\dots+k_s = k}\prod_{j=1}^s\frac{1}{k_j!(k_j+j+\sigma(j)-1)!}\right]\left|e^{(ix-1)t^2}\times\ 2^{k+1}t^{2k+1}\right| dt<\infty
    \label{eq:DCTrho}
\end{equation}
which can be seen by noting that $|e^{(ix-1)t^2}| = e^{-t^2}$ so that the integral in \eqref{eq:DCTrho} is equal to the integral in \eqref{Besselprodintegral} for a fixed $\sigma$ and $x=0$, which is finite by the bound for modified Bessel functions in \eqref{eq:besselbound}. Hence, by Fubini's theorem, we get the desired interchange of summation and integration. The remaining integrals are standard and an explicit evaluation yields:
\begin{equation}
    \rho^{(s)}(x) =  \mathsf{V}^{(s)}\frac{1}{2\pi}\sum_{\sigma \in \mathfrak{S}(s)}\sgn(\sigma)\sum_{k=0}^{\infty}\Re\left(\left[\sum_{k_1+\dots+k_s = k}\prod_{j=1}^s\frac{1}{k_j!(k_j+j+\sigma(j)-1)!}\right]k!\left(\frac{2}{1-ix}\right)^{k+1}\right).
\end{equation}
Finally, we can plug in the value of the constant and rewrite the sum over $\mathfrak{S}(s)$ as a determinant to obtain the expression in \eqref{eq:generalrho}. 

To compute the moments $\mathbb{E}\left(|\mathbf{X}(s)|^{2h}\right)$, and thus $\mathcal{R}(s,h)$ by the relation \eqref{eq:Rdefinition}, we temporarily restrict to $h \in (-\frac{1}{2},0)$. Now, for $h\in(-\frac{1}{2},0)$ we have, for all $k\geq 1$, that:
\begin{equation}
    \int_0^\infty \left|\Re\left(\frac{x^{2h}}{(1-ix)^{k+1}}\right)\right|dx\le\int_0^\infty \frac{x^{2h}}{\left({\sqrt{1+x^2}}\right)^{k+1}}dx\le\int_0^\infty \frac{x^{2h}}{{\sqrt{1+x^2}}}dx<\infty
    \label{eq:fubiniineq}
\end{equation}
and hence:
\begin{equation*}
    \sum_{k=0}^{\infty}\left(\left[\sum_{k_1+\dots+k_s = k}\prod_{j=1}^s\frac{1}{k_j!(k_j+j+\sigma(j)-1)!}\right]k!2^{k+1}\int_0^\infty \left|\Re\left(\frac{x^{2h}}{(1-ix)^{k+1}}\right)\right|\right)<\infty,
\end{equation*}
where the finiteness of the sum is seen by using the inequalities in \eqref{eq:fubiniineq} and comparing to the infinite sum for $\rho^{(s)}(0)$. Now, we simply apply Fubini's theorem to integrate the infinite series for $\rho^{(s)}(x)$ term-by-term, using the following evaluation from \cite[3.194.3]{IntegralsBook}:
\begin{equation*}
     \int_0^\infty\frac{x^{2h}}{(1-ix)^k}dx = -i\pi e^{i\pi h} \frac{(-2h)_{k-1}}{(k-1)!\sin(2\pi h)}.
\end{equation*}
Thus, we obtain the equality:
\begin{equation}
\cos(\pi h)\mathbb{E}\left(|\mathbf{X}(s)|^{2h}\right) = \mathsf{V}^{(s)}\left(\sum_{k=0}^\infty \left[\sum_{k_1+\ldots+k_s=k}\det\left[\frac{1}{(k_i+i+j-1)!}\right]_{i,j=1,\ldots,s}\prod_{j=1}^s \frac{1}{{k_j}!}\right] (-2h)_k 2^k\right) \label{eq:explicitE}
\end{equation}
for $h  \in (-\frac{1}{2},0)$. Now, letting $M$ be a compact subset of the strip $D:=\{z\in \mathbb{C}:  -\frac{1}{2} < \Re(z) < s+\frac{1}{2}\}$, and noting that there exists a positive constant $\alpha_M$ such that $|(-2h)_k|<\alpha_M k!$,  for all $h \in M$ whenever $k$ is large enough (this can be seen immediately via Stirling's approximation), we see that, by comparison with the sum in \eqref{eq:generalrho}, which converges for $x=0$, the infinite sum on the right hand-side of  \eqref{eq:explicitE} converges uniformly on compact subsets of $D$, an thus it is analytic in $h$ on $D$. Note also that: 
\begin{equation}
 \mathbb{E}\left(|\mathbf{X}(s)|^{2h}\right) = \int_{-\infty}^\infty |x|^{2h} \rho^{(s)}(x) dx 
\end{equation}
is an analytic function of $h$ for $h \in D$, as can be seen by a combination of Fubini's and Morera's theorems. More precisely, let $\Gamma$ be a closed path contained in $D$. Setting $\alpha_1 := \inf_{h \in \Gamma}\Re (h)>-\frac{1}{2}$ and $\alpha_2 := \sup_{h \in \Gamma}\Re (h) < s+\frac{1}{2}$, we see that $\left||x|^{2h}\right| \le |x|^{2\alpha_1}$ whenever $|x|\le1$ and $\left||x|^{2h}\right| \le |x|^{2\alpha_2}$ whenever $|x|\ge 1$. Therefore, we have that $\left||x|^{2h}\right| \le |x|^{2\alpha_1} + |x|^{2\alpha_2}$ for all $x \in \mathbb{R}$ and $h \in \Gamma$, so that:
\begin{equation*}
    \int_{\Gamma}\int_{-\infty}^\infty \left||x|^{2h} \rho^{(s)}(x)\right| dx dh \le \int_{\Gamma}\int_{-\infty}^\infty \left(|x|^{2\alpha_1}+ |x|^{2\alpha_2}\right) \rho^{(s)}(x) dx dh<\infty,
\end{equation*}
where the finiteness of the double integral is justified by the fact that $\Gamma$ is a finite length path, and that the inner integral is finite by the computation above when $\alpha_1, \alpha_2 \in (-\frac{1}{2},0)$, and by finiteness of $\mathbb{E}\left(|\mathbf{X}(s)|^{2\alpha_i}\right)$ for $\alpha_i \in [0,s+\frac{1}{2})$ (see \cite{assiotispainleve}). Hence, we have that: \begin{equation}
    \int_{\Gamma}\int_{-\infty}^\infty |x|^{2h} \rho^{(s)}(x) dx dh=\int_{-\infty}^\infty \int_{\Gamma} |x|^{2h} \rho^{(s)}(x) dh dx=0,
\end{equation}
where we have used Cauchy's theorem. Hence, by Morera's theorem, we see that $\mathbb{E}\left(|\mathbf{X}(s)|^{2h}\right)$ is an analytic function of $h$ on $D$.
Thus, it is now clear that $\cos(\pi h)\mathbb{E}\left(|\mathbf{X}(s)|^{2h}\right)$ and the infinite sum on the right-hand side of \eqref{eq:explicitE} are two analytic functions of $h$ for $h \in D$, which agree for $h\in(-\frac{1}{2},0)$, and so the equality \eqref{eq:explicitE} holds for all $h \in D$.

We note that  \eqref{eq:generalrho} and \eqref{eq:momentsgeneral} simplify to the corresponding expressions given in Theorem \ref{thm:s=0} for $s=1,2$, which can be seen immediately for $s=1$, and for $s=2$ by using Vandermonde's Identity\footnote{It is not clear whether such a combinatorial simplification exists for $s\ge 3$. We note that also Winn, when computing the half-integer moments in Section 7 of \cite{Winn}, observed that some combinatorial structures seem to break down for $s\ge 3$.}:
\begin{equation*}
    \sum_{k=0}^{n} {n\choose k} {s\choose t+k}={n+s\choose n+t}.
\end{equation*}
The proof of Theorem \ref{thm:s=0} is now complete.
\end{proof}
\begin{rmk}
Note that from the formula \eqref{eq:x1moments} we may immediately recover the previously known (see \cite{Zeta4}) value $\mathcal{R}(1,1) = \frac{1}{12}$. Note also that by the equality \eqref{eq:explicitE}, valid for $\Re(h) \in \left(-\frac{1}{2},s+\frac{1}{2}\right)$, and using that the moments $\mathbb{E}\left(|\mathbf{X}(s)|^{2h}\right)$ are finite for $\Re(h) \in \left(-\frac{1}{2},s+\frac{1}{2}\right)$ (as seen in the proof of Theorem \ref{thm:s=0}), we deduce the seemingly non-trivial fact that the sum on the right-hand side of \eqref{eq:explicitE} vanishes for all $h \in \mathbb{N} - \frac{1}{2}$ in this range. Therefore, calculations for the moments $\mathcal{R}(s,h)$ for $h \in \mathbb{N}-\frac{1}{2}$ can be performed by applying L'H\^opital's rule to the relevant formulae in Theorem \ref{thm:s=0}. For instance, the value $\mathcal{R}\left(1,\frac{1}{2}\right) = \frac{e^2-5}{4\pi}$ (as calculated by Winn in \cite{Winn}) can be recovered:
\begin{multline}
    \mathcal{R}\left(1,\frac{1}{2}\right)= -\frac{1}{2\pi} \lim\limits_{h \to \frac{1}{2}}\frac{d}{dh}\left(\sum_{k=0}^\infty \frac{(-2h)_k}{k!(k+1)!}2^k\right) \\ = -\frac{1}{2\pi}\left(-2+ \sum_{k=2}^\infty \frac{1}{(k-1)k(k+1)!}2^{k+1}\right) = \frac{e^2-5}{4\pi}, \label{eq:lhospital}
\end{multline}
where in the penultimate equality we have used the following limit formula, obtained via the product rule:
\begin{equation}
    \lim\limits_{h \to \frac{1}{2}} \frac{d}{dh}\left((-2h)_k\right)= \begin{cases}
    2(k-2)! & \text{if} \ k \ge 2, \\
    -2 & \text{if} \ k = 1, \\
    0 & \text{if} \ k = 0.
    \end{cases}
\end{equation}
The interchange of limits in \eqref{eq:lhospital} can be justified by standard arguments using uniform convergence. We also remark that the specific values of \eqref{eq:x1moments} for $h= \frac{-1}{4},\frac{1}{4},\frac{3}{4},$ and $\frac{5}{4}$, which had not been computed before, can be expressed as a combination of $I_0(1)$ and $I_1(1)$:
\begin{equation}
\begin{aligned}
\mathcal{R}\left(1,-\frac{1}{4}\right) &= 2e(I_0(1)-I_1(1)),\\
\mathcal{R}\left(1,\frac{1}{4}\right) &= \frac{e}{3}\left(-I_0(1)+3I_1(1)\right),\\
\mathcal{R}\left(1,\frac{3}{4}\right) &= \frac{e}{30}\left(5I_0(1)-9I_1(1)\right), \\
\mathcal{R}\left(1,\frac{5}{4}\right) &= \frac{e}{140}(5I_0(1)-3I_1(1)).
\end{aligned}
\end{equation}
These values are taken from known special values of the confluent hypergeometric function. Finally, from \eqref{eq:s2moment} we can immediately recover the previously known (see \cite[\S 6.2.]{Zeta4}) values $\mathcal{R}(2,1) = \frac{1}{720}$ and $\mathcal{R}(2,2)=\frac{1}{6720}$. Again using L'H\^opital's rule we can recover the half-integer values:
\begin{equation}
\begin{aligned}
    \mathcal{R}\left(2,\frac{1}{2}\right) &= \frac{7}{180\pi}\left(\frac{15}{7}- {}_3F_3\left[\begin{matrix}
   \frac{9}{2},1,1 \\
    3,6,7 \end{matrix} \ ; \ 8\right]\right),\\
    \mathcal{R}\left(2,\frac{3}{2}\right) &= \frac{11}{3360\pi}\left(-\frac{28}{33}+ {}_3F_3\left[\begin{matrix}
   \frac{13}{2},1,1 \\
    5,8,9 \end{matrix} \ ; \ 8\right]\right).
\end{aligned}
\end{equation}
These were previously calculated using Maple by Winn \cite[\S 6.2.]{Winn}, who used combinatorial expressions valid only for half-integer parameters.
\label{rmk:halfint}
\end{rmk}
\begin{rmk}
Using the fact that for fixed $s$, the series in \eqref{eq:momentsgeneral} vanishes for half-integer $h$ in the range $ (0,s+\frac{1}{2})$ and noting the known values for integer $h$ again in this range (see \cite[eq. (4-46)]{Basor0}), we can write $\mathcal{R}(s,h)$ as:
\begin{equation}
    \mathcal{R}(s,h)=\frac{G(s+1)^2}{G(2s+1)}2^{-2h}\frac{1}{\cos(\pi h)}\sum_{k=0}^\infty a_k(s) (-2h)_k
    \label{eq:rationalcoef}
\end{equation}
where $a_0(s),\ldots, a_{2m}(s)$ can be simplified to rational functions of $s$ so that this rational expression of the first $2m+1$ coefficients in \eqref{eq:momentsgeneral} is valid for all $s\geq m$. For instance, for $s\ge 2$ there exists an expansion of $\mathcal{R}(s,h)$ in the form \eqref{eq:rationalcoef}, where the first $5$ terms of the series are explicitly given as follows:
\begin{multline}
\mathcal{R}(s,h)= \frac{G(s+1)^2}{G(2s+1)}2^{-2h}\frac{1}{\cos(\pi h)} \\\times \left(1+(-2h)_1+\frac{4s^2-2}{4s^2-1}\frac{(-2h)_2}{2}+\frac{4s^2-4}{4s^2-1}\frac{(-2h)_3}{3!}+\frac{(4s^2-8)^2+2}{(4s^2-1)(4s^2-9)}\frac{(-2h)_4}{4!}+\dots\right).
\end{multline}
Note also that the first $3$ coefficients in the expansion above are correct for all $s\ge 1$, but the simplifications for $a_3(s), a_4(s)$ are only valid for $s\ge 2$. It is tempting to try to find an expansion in the form of \eqref{eq:rationalcoef} such that $a_0(s),\ldots, a_{2m+1}(s)$ are all rational functions of $s$, valid for all $s\geq m$ where $m$ is a positive integer. However, if one proves that $\lim_{h\to m+\frac{1}{2}}\mathcal{R}(m,h)=\infty$, which we expect to be true (we have verified this for $m=1,2,3,4$ by evaluating the right-hand side of \eqref{eq:explicitE} at $h=m+\frac{1}{2}$ for $m=1,2,3,4$ and finding that it did not vanish), but do not pursue further in this paper\footnote{We note that for finite $N$, one can show by expanding the Haar measure in \eqref{eq:jointmoments} that $\lim_{h \to m+\frac{1}{2}}\mathcal{R}_N(m,h) = \infty$; see \cite{combinatorial1}.}, then one can see that there cannot be an expansion in the form of \eqref{eq:rationalcoef} valid for all $s\geq m$ with $a_0(s),\ldots, a_{2m+1}(s)$ rational functions of $s$. Indeed if such an expansion existed, then we would have, for all integer $s>m$, that:
\begin{align*}
    0 &= \frac{G(2s+1)}{G(s+1)^2}2^{2m+1}\cos\left(\left(m+\frac{1}{2}\right)\pi\right)\mathcal{R}\left(s,m+\frac{1}{2}\right)\\ &=a_0(s)(-2m-1)_0+\ldots+a_{2m+1}(s)(-2m-1)_{2m+1},
\end{align*}
so that since $a_0(s)+\ldots+a_{2m+1}(s)(-2m-1)_{2m+1}$ is a rational function of $s$ with infinitely many zeros, it is zero identically. Using the fact that the sum on the right hand-side of \eqref{eq:momentsgeneral} is analytic on $\Re(h)>-\frac{1}{2}$, this would imply that $\mathcal{R}(m,m+\frac{1}{2})<\infty$, which contradicts the assumption that $\lim_{h\to m+\frac{1}{2}}\mathcal{R}(m,h)=\infty$ (which still remains to be proven for $m\ge5$); hence, we have a contradiction and such an expansion cannot exist.
\end{rmk}

\begin{proof}[Proof of Corollary \ref{cor:sineprocess}]
The sine process $\mathcal{S}$ with kernel given by \eqref{eq:sinekernel} is obtained from the process $\mathbf{C}^{(0)}$ under the mapping $x \mapsto \frac{1}{y}$, see \cite[Theorem I]{BorodinOlshanskiErgodic}. The result now follows immediately from the definition \eqref{eq:Xsdefinition} of $\mathbf{X}(s)$ and Theorem \ref{s0explicit}.
\end{proof}
\begin{proof}[Alternative proof of \eqref{explicitphi} in the case $s=1$]We use the following version of Aomoto's integral formula \cite{Selberg}:
\begin{equation}
    a_{N,k}^{(\alpha)} := \int_{[0,\infty]^N}\prod_{j=1}^ky_j \prod_{j=1}^Ny_j^{\alpha-1}e^{-y_j}\Delta(\mathbf{y})^2 d\mathbf{y} = C_N^{(\frac{\alpha-1}{2})} \times \prod_{j=1}^k (\alpha + N-j),
\end{equation}
where $C_N^{(s)}$ is defined as in \eqref{eq:constant}. Thus, expanding the factor $\prod_{j=1}^N\left(y_j+\frac{t}{N}\right)$ in \eqref{eq:laguerre} yields, for $t \ge 0$:
\begin{equation}
\phi_N^{(1)}(t) = e^\frac{-t}{2}\sum_{r=0}^{N} {N\choose r} \left(\frac{t}{N}\right)^r \times a_{N,N-r}^{(2)} \times \frac{1}{C_N^{(1)}} = e^\frac{-t}{2}\sum_{r=0}^{N} \frac{(N-r+1)_r}{N^r}\frac{1}{r!(r+1)!}t^r. \label{eq:phi1elementary}
\end{equation}
Hence, taking the limit $N \to \infty$ in the last equality of \eqref{eq:phi1elementary}, and using the fact that the functions $\phi_N^{(1)}(t)$ are even, yields:
\begin{equation}
\phi^{(1)}(t) = e^{-\frac{|t|}{2}}\sum_{r=0}^{\infty} \frac{1}{r!(r+1)!}|t|^r = e^{-\frac{|t|}{2}}\frac{I_1\left(2\sqrt{|t|}\right)}{\sqrt{|t|}},
\end{equation}
for all $t \in \mathbb{R}$.
\end{proof}

\end{section}
\begin{section}{The Sum of Inverse Points of the Bessel Process}
\label{InverseBessel}
In this section we establish a connection between the Laplace transform of the sum of inverse points of the Bessel point process and the $\sigma-$Painlev\'e III' equation, in analogy to Theorem \ref{zetainftypainleve}. Let $\nu>-1$ and recall that the Bessel point process (with parameter $\nu$), that we denote by $\mathcal{P}^{(\nu)}$, is the  determinantal point process on $(0,\infty)$ with infinitely many points, whose correlation kernel is given by, see \cite{Forrester1}, also \cite{TracyWidom}:
\begin{equation}
    \mathcal{K}^{(\nu)}(x,y)=\frac{\sqrt{x}J_{\nu+1}(\sqrt{x})J_{\nu}(\sqrt{y})-\sqrt{y}J_{\nu+1}(\sqrt{y})J_{\nu}(\sqrt{x})}{2(x-y)}
\end{equation}
where $J_{\nu}$ denotes the Bessel function with parameter $\nu$. The Bessel point process is a fundamental object which appears as the universal scaling limit of the eigenvalues of random matrices at the hard edge, see for example \cite{KuijlaarsVanlessen}, \cite{Lubinsky}, \cite{RiderWaters}.

Then, we define\footnote{The fact that $\mathbf{Y}(\nu)$ is almost surely finite follows from the results of \cite{HuaPickrell10}.}:
\begin{equation}
    \mathbf{Y}(\nu)=\sum_{x\in \mathcal{P}^{(\nu)}} \frac{1}{x}.
\end{equation}
The random variable $\mathbf{Y}(\nu)$ plays a similar role to $\mathbf{X}(s)$ in the ergodic decomposition of another distinguished family of unitarily invariant probability measures on $\mathbb{H}(\infty)$, the inverse Laguerre measures; more precisely it is equal in distribution with the parameter $\gamma_1$, see \cite{HuaPickrell10}. The main result of this section is the following:

\begin{thm}
Let $\nu > -1$. Define
\begin{equation*}
    h^{(\nu)}(t):=\frac{\nu^2}{4}+\xi^{(\nu)}(t) 
\end{equation*}
where
\begin{equation*}
    \xi^{(\nu)}(t):=t\frac{d}{dt}\log  \psi^{(\nu)}(t)
\end{equation*}
and 
\begin{equation*}
   \psi^{(\nu)}(t):=\mathbb{E}\left(e^{-4t\mathbf{Y}(\nu)}\right).
\end{equation*}
Then, $h^{(\nu)}(t)$ is $C^\omega$ on $(0,\infty)$ and is a solution to a special case of the $\sigma-$Painlev\'e III' equation with one parameter for $t\in (0,\infty)$:
\begin{equation}
    \left(t\frac{d^2 h^{(\nu)}}{dt^2}\right)^2=4\left(\frac{dh^{(\nu)}}{dt}\right)^2\left(h^{(\nu)}-t\frac{dh^{(\nu)}}{dt}\right)+2\nu\frac{dh^{(\nu)}}{dt}+1
\label{eq:xiinfpainleve}
\end{equation}
Furthermore, we have the following boundary conditions:
\begin{align}
\begin{cases}
    h^{(\nu)}(0) = \frac{\nu^2}{4}, & \text{for } \ \nu>0, \label{eq:xiInfbc1}\\ & \\
    \left.\frac{d}{dt} h^{(\nu)}(t)\right|_{t=0} = -\frac{1}{\nu}, & \text{for}\ \nu>1.
\end{cases}
\end{align}
\label{BesselThm}
\end{thm}
Before discussing some of the relevant literature we give a brief outline of the strategy of proof which is similar to the one of Theorem \ref{zetainftypainleve}. Namely, we will consider the Laplace transforms of a sequence of random variables that converge, in distribution, to $\mathbf{Y}(\nu)$. We begin with some preliminaries.

Let $\mathbf{M}_N$ be an $N\times N$ random matrix taken from the Laguerre Unitary Ensemble (LUE) with parameter $\nu>-1$, having law:
\begin{align*}
 \text{const} \cdot \det(\mathbf{H})^{\nu} \exp\left(-\text{Tr}(\mathbf{H})\right)\mathbb{1}_{\{\mathbf{H}\in \mathbb{H}_{+}(N)\}}
 d\mathbf{H}
\end{align*}
where $d\mathbf{H}$ is the Lebesgue measure on $\mathbb{H}(N)$ and $\mathbb{H}_{+}(N)$ denotes the space of $N\times N$ positive-definite Hermitian matrices, and the constant is chosen so that this is a probability measure on $\mathbb{H}_+(N)$.
Then, the eigenvalues of $\mathbf{M}_N$ are distributed according to the following probability measure on $\mathbb{R}_{+}^N/\mathfrak{S}(N)$:
\begin{equation}
     \frac{1}{\widetilde{C}_N^{(\nu)}}\cdot{\Delta(\mathbf{x})}^2 \prod _{j=1}^{N} {x_j}^{\nu}e^{-x_j}dx_j
     \label{LUEeig}
\end{equation}
where:
\begin{equation}
    \widetilde{C}_N^{(\nu)}=\prod_{j=1}^N \Gamma(j)\Gamma(\nu+j).
\end{equation}
Via the transformation $\mathbf{M}_N \mapsto \frac{2}{\mathbf{M}_N}$, the LUE is transformed to the inverse Laguerre ensemble, whose eigenvalue distribution is given by the following probability measure on $\mathbb{R}_{+}^N/\mathfrak{S}(N)$:
\begin{equation}
\frac{1}{E_N^{(\nu)}}\cdot{\Delta(\mathbf{y})}^2 \prod _{j=1}^{N} {y_j}^{-\nu-2N}e^{-\frac{2}{y_j}}dy_j 
\label{inverseeig}
\end{equation}
with:
\begin{equation}
    E_N^{(\nu)}=\prod_{j=1}^N -\frac{\left((j-\nu-2N)_{j-1}\right)^{2}(2j-\nu-2N-1)}{2^{2j-\nu-2N-1}\Gamma\left(-j+\nu+2N+1\right)(j-1)!},
\end{equation}
where $(a)_j$ denotes the Pochhammer symbol given by $(a)_j:=\prod_{i=1}^{j}(a+i-1), \ (a)_0:=1$.

Then, similarly to the Hua-Pickrell case, by a combination of the results from \cite{BorodinOlshanskiErgodic} (the existence of the limit) and \cite{HuaPickrell10} (the identification of the limit with $\mathbf{Y}(\nu)$) we have that:
\begin{equation}
   \sum_{i=1}^N \frac{2}{Nx_i}  \overset{d}{=}\sum_{j=1}^N \frac{y_j}{N} \xrightarrow[N\rightarrow\infty]{d} 8\mathbf{Y}(\nu).
\end{equation}
where $\left(x_1,x_2, \ldots, x_N\right)$ are distributed according to the probability measure in \eqref{LUEeig}, whereas $\left(y_1,y_2,\ldots, y_N\right)$ are distributed according to the probability measure in \eqref{inverseeig}.
Then, if we let
\begin{equation*}
     \psi_{N}^{(\nu)}(t)=\mathbb{E}\left(e^{-t\sum_{j=1}^N\frac{1}{Nx_j}}\right),
\end{equation*}
where the expectation is taken with respect to probability measure in \eqref{LUEeig}, we have that:
\begin{equation}
     \psi_{N}^{(\nu)}(t) \xrightarrow{N\rightarrow\infty} \psi^{(\nu)}(t)
\end{equation}
for all $t\in [0,\infty)$, with the convergence being uniform on compacts. Finally, we define:
\begin{equation*}
   \xi_{N}^{(\nu)}(t)=t\frac{d}{dt}\log \psi_{N}^{(\nu)}(t).
\end{equation*}

Now, regarding the relevant literature, as far as we are aware, the first time a connection between $\psi_N^{(\nu)}$ and Painlev\'e equations was established was in the physics literature \cite{OsipovKanzieper} by Osipov and Kanzieper. Then, Chen and Its established this result rigorously in \cite{HankeldetLagu} using the ladder operator method and an alternative proof from the viewpoint of integrable systems theory was given by Mezzadri and Simm in \cite{MezzadriSimm}. Later, a detailed study of the large $N$ limit of $\psi_N^{(\nu)}$ for $\nu>0$ (without the identification of the limit with the Laplace transform $\psi^{(\nu)}$ of $\mathbf{Y}(\nu)$ proven in \cite{HuaPickrell10}) was performed using Riemann-Hilbert problem methods in \cite{Xu1}, \cite{Xu2}, see also Section 3 in \cite{DaiForresterXu}. It would be possible, for $\nu>0$, to extract the required ingredients we need from \cite{Xu1}, \cite{Xu2}, \cite{DaiForresterXu} so that along with the discussion above we establish Theorem \ref{BesselThm} (for $\nu>0$). However, since we want to cover the entire range of parameter values, $\nu>-1$, we instead adapt the relatively short and elementary arguments given in Section 2 to the present setting.

Our starting point is the following proposition, the analogue of Proposition \ref{zeta_Npainleve}, which as mentioned above is essentially due to the works \cite{HankeldetLagu}, \cite{MezzadriSimm}, \cite{OsipovKanzieper}. Namely, by making the change of variables $t\mapsto tN$ and some algebraic manipulations it reduces to Theorem 3 of \cite{HankeldetLagu}. This theorem in \cite{HankeldetLagu} is stated only for $\nu>0$ \footnote{For the same reasons explained in the discussion before the statement of Proposition \ref{prop:HankelPainleve1}.} but it can be extended to $\nu>-1$ via an analytic continuation argument identical to the one presented in the proof of Proposition \ref{prop:HankelPainleve1}. For an alternative approach to this result see \cite{MezzadriSimm} and also the physics paper \cite{OsipovKanzieper}.
\begin{prop}
Let $\nu >-1$. Then, for $t\in (0,\infty)$, $\xi_{N}^{(\nu)}(t)$ is a solution to a particular Painlev\'e equation:
\begin{equation}
    \left(t\frac{d^2\xi_{N}^{(\nu)}}{dt^2}\right)^2=-4t\left(\frac{d\xi_{N}^{(\nu)}}{dt}\right)^3+\left(\nu^2+4\xi_{N}^{(\nu)}+\frac{4t}{N}\right)\left(\frac{d\xi_{N}^{(\nu)}}{dt}\right)^2+\left(2\nu-\frac{4}{N}\xi_{N}^{(\nu)}(t)\right)\frac{d\xi_{N}^{(s)}}{dt}+1.
    \label{eq:painlevexi}
\end{equation}

\end{prop}

Now, we prove a series of propositions that will allow us to deduce Theorem \ref{BesselThm}.

\begin{prop}
Let $\nu>-1$. Then, there exist holomorphic functions  $g_N$ for $N=1,2,3,\dots$ and $g$ on $\{ z \in \mathbb{C}: \Re z  > 0 \}$, with $g_N(0) = g(0)=1$, such that
\begin{align}
    \psi_N^{(\nu)} = \left.g_N\right|_{[0,\infty)} && \text{and} &&  \psi^{(\nu)} = \left.g\right|_{[0,\infty)}.
\end{align}
Moreover, for $\Re z > 0, \ p=1,2,\dots$ we have:
\begin{equation}
\lim\limits_{N\to\infty} \frac{d^p}{dz^p}g_{N}(z) = \frac{d^p}{dz^p}g(z).
\end{equation}
\label{Laguerreanalytic}
\end{prop}
\begin{proof}
We use the same sequence of arguments we used in the proof of Proposition \ref{prop:analytic} by noting that $\left|e^{-\frac{z}{Nx}}\right|<1$ uniformly on $\{z \in \mathbb{C} : \Re z > 0 \}$, $ x>0$ and $N\geq 1$.
\end{proof}
\begin{prop}
Let $\nu>-1$. Then, there exists an exchangeable sequence of random variables $\left\{\mathsf{e}_i\right\}_{i=1}^\infty$\footnote{The sequence $\left\{\mathsf{e}_i\right\}_{i=1}^\infty$ is simply given by the diagonal elements of an infinite inverse Laguerre distributed random matrix with parameter $\nu$ on $\mathbb{H}(\infty)$. This probability measure on $\mathbb{H}(\infty)$ was constructed in \cite{HuaPickrell10}. } having the following Inverse-Gamma distribution on $(0, \infty)$:
\begin{equation}
    \frac{2^{\nu+1}}{\Gamma(\nu+1)}x^{-\nu-2}e^{-\frac{2}{x}}dx,
    \label{eq:lawe}
\end{equation}
such that 
\begin{equation}
 \sum_{i=1}^N \frac{2}{x_i} 
 \overset{d}{=}
 \sum_{i=1}^N y_i
 \overset{d}{=}
 \sum_{i=1}^N \mathsf{e}_i, \ \ \text{for \ all} \ N\ge 1, \end{equation}
where $\left(x_1,x_2, \ldots ,x_N\right)$ are distributed according to the probability measure \eqref{LUEeig} while $\left(y_1,y_2,\ldots ,y_N\right)$ are distributed according to the probability measure \eqref{inverseeig}.
\begin{proof}
We effectively follow the proof in \cite[Proposition 2.11]{assiotispainleve} verbatim, using the analogous results on the inverse Laguerre ensemble established in \cite{HuaPickrell10}, and noting that \eqref{eq:lawe} is the law of $y_1$ in the case $N=1$.
\end{proof}
\label{prop:exchange}
\end{prop}
\begin{prop}
Let $\nu > r-1\geq0$. Then, the sequence of random variables
\begin{equation}\label{InverseLaguerreSum}
    \left\{\bigg|\sum_{j=1}^N\frac{1}{Nx_j}\bigg|^r\right\}_{N\ge1}
\end{equation}
where $\left(x_1,x_2, \ldots, x_N\right)$ are distributed according to the measure in \eqref{LUEeig}, is uniformly integrable.
\begin{proof}
To show uniform integrability of $\left\{\bigg|\sum_{j=1}^N\frac{1}{Nx_j}\bigg|^r\right\}_{N\ge1}$ we simply show uniform boundedness of a higher moment:
for all $ r \in [1,\nu+1)$, there exists $k\in (r,\nu+1)$ so that, by Jensen's inequality: 
\begin{equation}
    \sup_{N\ge 1}\mathbb{E}\left(\bigg|\sum_{j=1}^N\frac{1}{Nx_j}\bigg|^k\right) \le 2^{-k}\mathbb{E}[|\mathsf{e_1}|^k] = \frac{\Gamma(\nu+1)}{2^{2k}\Gamma(\nu-k+1)} < \infty,
\end{equation}
where we have used Proposition \ref{prop:exchange} for both the bound and the equality. This implies that the sequence $\left\{\bigg|\sum_{j=1}^N\frac{1}{Nx_j}\bigg|^r\right\}_{N\ge1}$ is uniformly integrable for all $ r \in [1,\nu+1)$.
\end{proof}
\label{prop:uniformInt}
\end{prop}
Now, we are finally in a position to prove the main result of this section:
\begin{proof}[Proof of Theorem \ref{BesselThm}]
Using an argument similar to the proof of Proposition \ref{prop:analytic} we get the convergence of derivatives for $p=0,1,2,\ldots$:
\begin{equation}
    \frac{d^p\psi_{N}^{(\nu)}}{dt^p}\xrightarrow{N\rightarrow\infty} \frac{d^p\psi^{(\nu)}}{dt^p}
\end{equation}
for all $ t \in (0,\infty)$. Now, for $\nu>-1$, it was proven in \cite[Proposition 7.2]{HuaPickrell10} that $\mathbf{Y}(\nu)$ is finite almost surely. Moreover, it is clear from the definition of the eigenvalue density \eqref{LUEeig} that $\mathbb{P}\left( \sum_{j=1}^N \frac{1}{Nx_j}  <  1\right)$ is strictly positive for all $N$. Therefore, an application of Markov's inequality yields that there are constants $m_N^{(\nu)} > 0$, $m^{(\nu)}>0$ and $M^{(\nu)} > 0$ such that
\begin{align}
    \psi_N^{(\nu)}(t) \ge m_N^{(\nu)} e^{-t} > 0 \ \ \text{and} \ \  \psi^{(\nu)}(t) \ge m^{(\nu)}e^{-M^{(\nu)} t}>0, && \text{for \ all} \ t \in (0,  \infty).\label{eq:psiinequality}
\end{align}
Therefore, $\xi_N^{(\nu)}$ and its derivatives are well-defined and we can take the limits for $t \in (0,\infty), p = 0,1,2,\dots$:
\begin{equation}
        \frac{d^p}{dt^p}\xi_{N}^{(\nu)}(t)\xrightarrow{N\rightarrow\infty} \frac{d^p}{dt^p}\xi^{(\nu)}(t).
\end{equation}
Then, we simply take the limit as $N\rightarrow \infty$ of \eqref{eq:painlevexi} and substitute $h^{(\nu)}(t)=\frac{\nu^2}{4}+\xi^{(\nu)}(t)$ to get the desired Painlev\'e equation \eqref{eq:xiinfpainleve}. To show that $\xi^{(\nu)}(t)$ is $C^\omega$ on $(0,\infty)$, we simply use the fact that $\psi^{(\nu)}(t)>0$ on $(0,\infty)$ and apply the same sequence of arguments as in the proof of Theorem \ref{zetainftypainleve}.

For the boundary conditions, letting $\left(x_1,x_2, \ldots, x_N\right)$ be distributed according to the probability measure in \eqref{LUEeig},  we simply compute, using Proposition \ref{prop:exchange}: 
\begin{equation}
  \mathbb{E}\left(\sum_{j=1}^N\frac{1}{Nx_j}\right) =\mathbb{E}\left[\frac{\mathsf{e_1}}{2}\right] = \frac{1}{\nu}  
\end{equation}
for $\nu > 0$ and note that by Proposition \ref{prop:uniformInt} we have that
\begin{equation}
    \mathbb{E}\left(\left(\sum_{j=1}^N\frac{1}{Nx_j}\right)^2\right) < \infty,
\end{equation}
for $\nu > 1$. Hence, we get the boundary conditions:
\begin{equation}
    \begin{cases}
      \xi_N^{(\nu)}(0) = 0, & \text{for}  \ \nu>0,\label{eq:xiNbc1}\\
      & \\
   \left.\frac{d}{dt}\xi_N^{(\nu)}(t)\right|_{t=0} = -\frac{1}{\nu}, & \text{for} \ \nu>1, 
  
\end{cases}
\end{equation}
for all $N \ge 1$. Now, since by Proposition \ref{prop:uniformInt} we have that $\nu > r-1\geq 0$ implies that the sequence $\left\{\left|\sum_{j=1}^N\frac{1}{Nx_j}\right|^r\right\}_{N\ge1}$ is uniformly integrable, arguing as in the proof of Theorem \ref{zetainftypainleve} and using \eqref{eq:psiinequality} we establish:
\begin{equation}
 \frac{d^p}{dt^p}\xi_{N}^{(\nu)}(t)\xrightarrow{N\rightarrow\infty} \frac{d^p}{dt^p}\xi^{(\nu)}(t)
\end{equation}
for $t \in [0, \infty)$, $p=0$ when $\nu>0$ and $t \in [0, \infty)$, $p\in \{0,1\}$ when $\nu>1$. Thus, the desired boundary conditions are obtained by taking the limit $N \to \infty$ in \eqref{eq:xiNbc1}.
\end{proof}

\end{section}

\bigskip
\noindent
{\sc School of Mathematics, University of Edinburgh, James Clerk Maxwell Building, Peter Guthrie Tait Rd, Edinburgh EH9 3FD, U.K.}\newline
\href{mailto:theo.assiotis@ed.ac.uk}{\small theo.assiotis@ed.ac.uk}

\noindent
{\sc Mathematical Institute, Andrew Wiles Building, University of Oxford, Radcliffe
Observatory Quarter, Woodstock Road, Oxford, OX2 6GG, UK.}\newline
\href{mailto:benjamin.bedert@sjc.ox.ac.uk}{\small benjamin.bedert@sjc.ox.ac.uk}

\noindent
{\sc Mathematical Institute, Andrew Wiles Building, University of Oxford, Radcliffe
Observatory Quarter, Woodstock Road, Oxford, OX2 6GG, UK.}\newline
\href{mailto:mustafa.gunes@st-hildas.ox.ac.uk}{\small mustafa.gunes@st-hildas.ox.ac.uk}

\noindent
{\sc Mathematical Institute, Andrew Wiles Building, University of Oxford, Radcliffe
Observatory Quarter, Woodstock Road, Oxford, OX2 6GG, UK.}\newline
\href{mailto:arun.soor@sjc.ox.ac.uk}{\small arun.soor@sjc.ox.ac.uk}

\end{document}